\documentclass[review,onefignum,onetabnum]{siamart171218}

\usepackage{geometry}
\usepackage{commath}
\usepackage{enumitem}
\usepackage{tcolorbox}

\usepackage{multirow}

\usepackage{lipsum}
\usepackage{amsfonts}
\usepackage{graphicx}
\usepackage{epstopdf}
\usepackage{algorithmic}
\usepackage{subcaption}
\usepackage{mathtools}
\usepackage{amsmath}
\ifpdf
  \DeclareGraphicsExtensions{.eps,.pdf,.png,.jpg}
\else
  \DeclareGraphicsExtensions{.eps}
\fi

\newsiamremark{remark}{Remark}
\newsiamremark{hypothesis}{Hypothesis}
\crefname{hypothesis}{Hypothesis}{Hypotheses}
\newsiamthm{claim}{Claim}

\headers{Deep Tangent Bundle Method}{H. Wu and H. Zhou}

\title{Deep Tangent Bundle (DTB) Method: a deep neural network approach to compute solutions of PDEs 
}

\author{Hao Wu\thanks{Charlotte, NC, USA
 (\email{hwu406@gmail.com}).}
  \and
 Haomin Zhou\thanks{School of Mathematics, Georgia Institute of Technology, Atlanta, GA, USA(\email{hmzhou@gatech.edu}).}
 }

\usepackage{amsopn}

\makeatletter
\newcommand*{\addFileDependency}[1]{
  \typeout{(#1)}
  \@addtofilelist{#1}
  \IfFileExists{#1}{}{\typeout{No file #1.}}
}
\makeatother

\newcommand*{\myexternaldocument}[1]{%
    \externaldocument{#1}%
    \addFileDependency{#1.tex}%
    \addFileDependency{#1.aux}%
}

\ifpdf
\hypersetup{
  pdftitle={The Deep Tangent Bundle Method: a general approach to improve Neural Network solvers to PDE},
  pdfauthor={robot}
}
\fi

\myexternaldocument{ex_supplement}

\begin{document}

\maketitle

\begin{abstract}
We develop a numerical framework, the Deep Tangent Bundle (DTB) method, that is suitable for computing solutions of evolutionary partial differential equations (PDEs) in high dimensions. The main idea is to use the tangent bundle of an adaptively updated deep neural network (DNN) to approximate the vector field in the spatial variables while applying the traditional schemes for time discretization. The DTB method takes advantage of the expression power of DNNs and the simplicity of the tangent bundle approximation. It does not involve nonconvex optimization. Several numerical examples demonstrate that the DTB is simple, flexible, and efficient for various PDEs of higher dimensions.  

\end{abstract}


\section{Introduction}

In this paper, we focus on numerical simulations of the initial value problems of evolutionary partial differential equations (PDEs) in high dimensions given as
\begin{subequations}
\label{eq: evolutional pde}
\begin{align}
    &\partial_t u(t, z) = F[u](t, z),\\
    &u(0, z) = \phi(z),
\end{align}
\end{subequations}
where $F[\cdot]$ represents the (possibly nonlinear) differential operator, and $\phi(z)$ is the initial condition.

In recent years, numerous deep neural network (DNN)-based approaches have been proposed to solve PDEs \cite{han2017deep, ruthotto2020machine,gaby2023neural,zhang2023transnettransferableneuralnetworks,liu2024naturalprimaldualhybridgradient, HARDWICK2025114273, Lu_2021, li2021fourierneuraloperatorparametric}. As these methods are often sample-based and mesh-free, they are particularly suitable for high-dimensional settings. A prominent class of methods parameterizes the solution globally in time using a single network \( \tilde{f}_{\theta}(t, z) \) with input variables $(t, z)$ and $\theta$ are the DNN parameters. Examples like Physics-Informed Neural Networks (PINNs) \cite{raissi2019physics}, Deep Ritz Method (DRM) \cite{yu2018deep}, and Weak Adversarial Networks (WANs) \cite{zang2020weak, bao2020numerical} formulate the PDE as an optimization problem. They employ gradient-based training to minimize the residual \( \partial_t \tilde{f}_{\theta} - F[\tilde{f}_{\theta}] \) in the sense of a strong or weak solution, while enforcing the boundary conditions using penalty terms. It has been demonstrated in many studies that these methods are highly flexible with remarkable potential in addressing some difficult simulations. However, the nonlinear structure of DNNs, which makes the optimization non-convex, can introduce challenges such as training instability and convergence difficulties.

An alternative, designed specifically for evolutionary PDEs, is the so-called local-in-time (LIT) framework such as Neural Galerkin method \cite{BRUNA2024112588, berman2023randomizedsparseneuralgalerkin}, Time-Evolving Natural Gradient (TENG) method \cite{chen2024teng} and the neural parametric method \cite{doi:10.1137/20M1344986}. Instead of using global-in-time parameterization, the LIT methods first approximate the initial condition with a DNN $f_{\theta}(z)$, dependent only on the spatial variable. The evolution is then captured by solving for the trajectory of the parameters $\theta(t)$, which is typically formulated as a finite-dimensional ordinary differential equation (ODE) system. This allows the use of classical numerical schemes to solve the resulting ODEs, often in a training-free manner, so that the efficiency and accuracy are enhanced.

As a particular form of the LIT methods, it was shown in Parameterized Wasserstein Hamiltonian Flow (PWHF) \cite{doi:10.1137/23M159281X} that the gradients \( \partial_{\theta} f_{\theta} \) spans a \( L^2 \) subspace \( \mathcal{T}_{\theta} \), and that the evolution of $\theta (t)$ is closely related to the metric tensor defined in the tangent space $\mathcal{T}_{\theta}$. However, convergence analysis and numerical experiments in PWHF also reveal a possible degeneration issue associated with the metric. Specifically, as \( \theta (t)\)  evolves, the metric tensor may become ill-conditioned or singular, resulting in two severe consequences. First, it leads to a larger approximation error, as the expressive power of the subspace \( \mathcal{T}_{\theta} \) is compromised. Second, it induces stiffness in the ODE system for $\theta(t)$, necessitating prohibitively small time steps to avoid instability in the function approximation due to the high nonlinearity of the DNN. Since the trajectory of $\theta(t)$ is driven by the PDE itself, this degeneracy is an inherent challenge that limits the potential application of such methods.

Motivated by these observations, we propose to use the tangent bundle of the DNNs to directly represent the vector field $F[u]$ in the PDEs in this work.  In particular, we develop a scheme that projects $F[u]$ onto the tangent space of $f_{\theta}$ and then advances the solution in time with numerical integrators. Meanwhile, the parameters $\theta$ are adapted dynamically when necessary. Our approach retains the properties of being mesh‐free, having high‐dimensional capability of DNN methods while mitigating ill-condition and stability issues that arise in the existing LIT schemes.

This work shares some similarities with the Random Feature Method (RFM) \cite{Chen2024} and the Extreme Learning Machine (ELM) method \cite{DONG2021114129, DONG2022111290}, as all of them use linear combinations of feature functions to approximate PDE solutions. The RFM, for example, constructs features using randomly sampled parameters, which enhances its computation efficiency and scalability. Unlike the RFM and ELM that rely on fixed randomly generated features, the proposed method constructs a problem-informed adaptive tangent bundle that aims to capture the PDE structure more directly. This may help find features with comparable or fewer parameters, but stronger expressiveness to the solutions, which may lead to more accurate and stable approximations, especially when dealing with complicated high-dimensional nonlinear PDEs.

The paper is organized as follows. Sections \ref{sec: func approx intro} and \ref{sec: dtb pde solve intro} introduce the DTB methodology and detail the algorithm design. Section \ref{sec: error estimate} provides a systematic analysis of its error sources. Section \ref{sec: numerical method} presents numerical results that demonstrate the performance of the method in various tasks.

\section{The DTB method}\label{sec: alg design and error analysis}
To address these issues for evolutionary PDEs with nonlinearity in high dimensions, we propose the Deep Tangent Bundle (DTB) method, whose design is guided by two core principles.

\begin{enumerate}
    \item The tangent bundle of a DNN is used to spatially approximate $F[u]$, while traditional temporal schemes can be used to compute the solution.
    \item The DNN parameters used to generate the tangent bundle are adaptively updated to improve accuracy and efficiency.
\end{enumerate}

A crucial insight underlying our approach, inspired by the error analysis of the parameterized Wasserstein Hamiltonian flow (PWHF) shown in \cite{doi:10.1137/23M159281X}, is that the orthogonal projection of the right-hand side of the PDE onto the DNN tangent space provides an upper bound quantifying the accuracy of the approximation. In other words, the evolutionary PDE solution can be accurately computed if its temporal rate of change can be efficiently approximated in the DNN tangent space. 

In this section, we introduce the DTB method step by step. Section \ref{sec: func approx intro} discusses how to use DTB to approximate functions, which is the backbone of the first principle.  Section \ref{sec: dtb pde solve intro} explains how to construct the PDE solution using DTB, with a guideline for adaptively updating the DNN parameters. We provide an error estimate for the DTB method and highlight its key advantages in Section \ref{sec: error estimate}.

\subsection{The DTB function approximation}\label{sec: func approx intro}
We consider how to approximate a real valued function \( g: \mathbf{R}^d \to \mathbf{R}^q \). We denote the DNN as \(f_\theta\) with \(\theta\in \mathbf{R}^m\), typically represented as compositions of activation functions and linear transformations. For example, a multilayer perceptron (MLP) with $L$ layers can be defined as:
\begin{align}
    f_{\theta}(z) &= W_{L} \sigma \left( W_{L-1} \sigma \left( \cdots \sigma \left( W_1 z + b_1 \right) \right) + b_{L-1} \right)+b_L,
\end{align}
where $\theta=\{(W_i, b_i)\}_{i=1}^L$, $W_i$ and $b_i$ are matrices or vectors with proper dimensions, $\sigma$ is a non-linear activation function applied element-wise. To simplify the presentation, we use lower subscripts to denote elements of vectors (e.g., \(\theta_i\) is the \(i\)-th element of the vector \(\theta\)), while superscripts indicate different vectors or parameter states (e.g., \(\theta^k\) represents the parameters at step \(k\)).

Unlike the existing DNN based function approximation, which aims to find the optimal parameters $\theta$ to minimize the distance between $g$ and $f_{\theta}$, the \textbf{DNN tangent bundle}, abbreviated DTB and denoted by $\mathcal{B}(f, \theta)=\text{span}\{\frac{\partial}{\partial\theta_i}f_{\theta}:1\leq i\leq m\}$, offers an alternative approach to function approximation by defining
\begin{align}
    \label{eq: DTB notation}
    TB(f_{\theta}; v) = \frac{\partial}{\partial\theta}f_{\theta}(z) \cdot v,
\end{align}
where $v \in \mathbf{R}^m$ are the DTB coefficients representing the linear combination coefficients. The best approximation is obtained when $v$ is taken as the optimal coefficients $\alpha$, namely $\alpha$ is the minimizer of the following least-squares problem, 
\begin{equation}
\label{eq: alpha def}
    \alpha = \arg \min_{v\in \mathbf{R}^m} \int \left\| \frac{\partial}{\partial \theta}f_{\theta}(z)\cdot v - g(z) \right\|^2 \, dz.
\end{equation}
This is a semi-convex optimization problem whose solution is given by 
\begin{align}
\label{eq: G alpha sol}
    \alpha = G(\theta)^{\dagger}P(\theta; g),
\end{align}
where $G(\theta)\in \mathbf{R}^{m\times m}$ is the metric tensor given in \cite{doi:10.1137/23M159281X}
\begin{align}
    G(\theta) = \int \frac{\partial}{\partial\theta}f_{\theta}(z)^\top \frac{\partial}{\partial\theta}f_{\theta}(z)dz,
\end{align}
and $P(\theta; \cdot): \mathcal{L}^2(\mathbf{R}^d; \mathbf{R}^q)\rightarrow \mathbf{R}^m$ is the projection 
\begin{align}
    P(\theta; g)=\int  \frac{\partial}{\partial\theta}f_{\theta}(z)^\top g(z) dz.
\end{align}

Introducing a projection operator in $L^2$ space 
\begin{align}\label{eq: func proj op def}
    \mathcal{K}_\theta[\cdot](z) = \frac{\partial}{\partial\theta}f_{\theta}(z)G(\theta)^{\dagger}P(\theta; \cdot), 
\end{align}
we can rewrite the optimal DTB approximation of $g$ as 
\begin{align}
    \label{eq: func opt approx}
    \mathcal{K}_\theta[g](z) = \frac{\partial}{\partial\theta}f_{\theta}(z)\left(\int \frac{\partial}{\partial\theta}f_{\theta}(z)^\top \frac{\partial}{\partial\theta}f_{\theta}(z)dz\right)^{\dagger} \int  \frac{\partial}{\partial\theta}f_{\theta}(z)\cdot g(z) dz = TB(f_{\theta}; \alpha).  
\end{align}

A key distinction of the DTB approximation is that the variable in the optimization is the DTB coefficients $v$, not the DNN parameters $\theta$. Its close-form solution can be derived analytically. Meanwhile, its approximation power is supported by the existing DNN universal approximation theorem through the following proposition.

\begin{proposition}
Let \(f_\theta:\mathbf{R}^d \to \mathbf{R}^q\) be a deep neural network 
whose last layer is linear without bias.  Then  
\begin{align}
  f_\theta\;\in\;\mathrm{span}\bigl\{\partial_{\theta}f_\theta\bigr\}=\mathcal{B}\bigl(f;\theta\bigr),
  \quad \text{for arbitrary } \theta.    
\end{align}
\label{prop: approximation}
\end{proposition}
\begin{proof}
    Since the last layer is linear, there exists a DNN $\tilde{f}_{\tilde{\theta}}$ and vector $w$ such that $f_{\theta}(z)=\tilde{f}_{\tilde{\theta}}(z)\cdot w$ with $\theta = (\tilde{\theta}, w)\in \mathbf{R}^m, \tilde{\theta}\in \mathbf{R}^{m-m_1}, w\in \mathbf{R}^{m_1}, \tilde{f}_{\tilde{\theta}}(z)\in \mathbf{R}^{q\times m_1} $. Notice that $\frac{\partial}{\partial\theta}f_{\theta} = (\frac{\partial}{\partial\tilde{\theta}}\tilde{f}_{\tilde{\theta}}\cdot w, \tilde{f}_{\tilde{\theta}})$. Take DTB coefficient $v=(\textbf{0}, w)\in \mathbf{R}^{m}$, and we can check:
    \begin{align*}
        f_{\theta}(z)=\tilde{f}_{\tilde{\theta}}(z) \cdot w=  (\frac{\partial}{\partial\tilde{\theta}}\tilde{f}_{\tilde{\theta}}\cdot w, \tilde{f}_{\tilde{\theta}})\cdot (\textbf{0}, w)=\frac{\partial}{\partial\theta}f_{\theta}\cdot v\in\;\mathrm{span}\bigl\{\partial_{\theta}f_\theta\bigr\}.
    \end{align*}
\end{proof}



Although Proposition \ref{prop: approximation} gives an approach to transfer the results from the DNN approximation to the DTB, it provides little practical guidance. On the other hand, the DTB of a DNN, even if it is not optimally constructed, may act as an overly redundant basis and can still achieve reasonable approximation results. We illustrate this observation using several numerical experiments in Section \ref{sec: numerical method}.

The closed-form solution \eqref{eq: G alpha sol} leads to an algorithm which is given in Algorithm \ref{alg: G-DTB}. It computes the approximation using the $G$ metric tensor. We call it the $G$-form DTB method. 

\begin{algorithm}[H]
\caption{$G$-form DTB for function approximation}
\label{alg: G-DTB}
\begin{algorithmic}[1]
\STATE{\textbf{Input}: Given a target function $g$.}
\STATE{Generate samples $\{z_i\}_{i=1}^{n}$ for evaluation;}
\STATE{Choose a (accumulated/stacked) DTB base $f_{\theta}$;}
\STATE{Formulate the empirical DTB metric operator $G(\theta)$ as well as projection vector $P(\theta; g)$ through samples;}
\STATE{Apply linear system solver to solve $\alpha$ from $\alpha = G(\theta)^{\dagger}P(\theta; g)$;}

\STATE{\textbf{Output}: Solution pair $(\theta, \alpha)$ and the DTB approximation function $TB(f_{\theta}(\cdot); \alpha) = \frac{\partial}{\partial\theta}f_{\theta}(\cdot) \cdot \alpha$. }
\end{algorithmic}
\end{algorithm}

Since \eqref{eq: alpha def} is a least-squares problem, many existing algorithms can be adopted. In fact, the $G$-form DTB algorithm is based on the normal equation strategy. An alternative is to solve it by singular value decomposition (SVD), i.e. using the SVD to solve the linear system defined by the Jacobian matrix,
\begin{align}
\label{eq: J alpha solution}
    \frac{\partial}{\partial\theta}f_{\theta}(z) \cdot \alpha = g(z),
\end{align}
we call this formulation the $J$-form DTB and provide its details in Algorithm \ref{alg: J-DTB}. 

\begin{algorithm}[H]
\caption{$J$-form DTB for function approximation}
\label{alg: J-DTB}
\begin{algorithmic}[1]
\STATE{\textbf{Input}: Given a target function $g$.}
\STATE{Generate samples $\{z^i\}_{i=1}^{n}$ for evaluation;}
\STATE{Choose a (accumulated/stacked) DTB base $f_{\theta}$;}
\STATE{Compute the Jacobian $J(\theta;z^i)=\frac{\partial}{\partial\theta}f_{\theta}(z^i)$ for each sample $z^i$; formulate the Jacobian matrix $\textrm{Jac}(\theta)=[J(\theta;z^1)^\top, \dots, J(\theta;z^n)^\top]^\top$ and the target vector $\vec{g}=[g(z^1)^\top, \dots, g(z^n)^\top]^\top$}
\STATE{Apply SVD-based linear system solver to solve $\alpha$ from $\alpha = \textrm{Jac}(\theta)^{\dagger}\vec{g}$;}

\STATE{\textbf{Output}: Solution pair $(\theta, \alpha)$ and the DTB approximation function $TB(f_{\theta}(\cdot); \alpha) = \frac{\partial}{\partial\theta}f_{\theta}(\cdot) \cdot \alpha$. }
\end{algorithmic}
\end{algorithm}

In an ideal scenario where an exact solution can be found, both the $G$-form and the $J$-form yield the same result. However, in practice, they can differ in terms of effectiveness, accuracy, and stability. Based on numerical analysis and experimental findings, the $G$-form leads to a linear system with a smaller dimension when the number of samples exceeds the number of DTB basis functions (the number of parameters in DNN). An effective implementation of the $G$ metric tensor as a linear operator using GPUs has been developed in \cite{doi:10.1137/23M159281X}, making the $G$-form potentially more computationally efficient.

On the other hand, the $J$-form has a smaller condition number, and experimental results indicate that the $J$-form provides better approximation accuracy, particularly when random samples are used to construct the linear system.
With the advancement of machine learning libraries, the Jacobian matrix 
$\textrm{Jac}(\theta)$ can now be computed efficiently. Consequently, the second choice, the $J$-form DTB method, has become a practical option.

Although choosing between the $G$-form and the $J$-form methods depends on the specific problem setting and computational constraints, both formulations offer complementary strengths. 

Since DTB is often redundant, various techniques can be used to enhance performance when solving \eqref{eq: G alpha sol}. For example, a random sparse subspace method was proposed in \cite{berman2023randomizedsparseneuralgalerkin}, which has demonstrated significantly improved efficiency and stability. Inspired by this study, a randomly selected subspace of $\mathcal{B}(f, \theta)$ can be used for the approximation in our context. The selection of a random subset \(\mathcal{B}_0\) is typically performed using a partial indexing approach. We randomly select \(l\) indices from the set \(\{1, 2, \ldots, m\}\), forming a subset \(s\). Then, we define \(\mathcal{B}_0\) as the span of derivatives \(\{\partial_{\theta_i} f_{\theta} : i \in s\}\). The corresponding metric tensor \(G(\theta)\) is given by the Gram matrix of these basis functions with respect to the \(L^2\) inner product,

\begin{equation}
\label{eq: sparse metric}
G_{ij}(\theta) = 
\begin{cases}
\displaystyle \int \partial_{\theta_i} f_\theta(z)^\top \, \partial_{\theta_j} f_\theta(z) \, dz, & \text{if } i,j \in s, \\[8pt]
0, & \text{if } i,j \notin s,
\end{cases}
\end{equation}
which is then used in finding the approximation of \eqref{eq: G alpha sol}.

\begin{remark}
The choice of DNN in the formulation \eqref{eq: alpha def} can be flexible. In principle, any DNN with $L^2$-integrable gradient with respect to $\theta $ can be selected, although different choices can lead to different results. An added feature for the DTB, which sets it apart from several other LIT methods, is its convenience in combining multiple DNNs to enhance the approximation power. 
In this case, the tangent bundle becomes the direct sum of the tangent bundles of these individual DNNs. We highlight that the simplicity and flexibility of merging multiple DNNs into a unified tangent bundle allows for greater expressiveness in approximating complex functions, and this becomes particularly important when adaptive strategies are needed in practice.  
\end{remark}

\subsection{The DTB method for solving evolution PDEs}\label{sec: dtb pde solve intro}
In this part, we first present the basic ideas in designing DTB solvers for evolutionary PDEs. Then we discuss the construction of temporal high-order DTB solvers. In \ref{sec: WF intro}, we apply the DTB solvers to the Wasserstein gradient and Hamiltonian flows.

\subsubsection{The basic ideas of DTB solvers}
Following the two principles given at the beginning of Section \ref{sec: alg design and error analysis}, we describe how the solution of \eqref{eq: evolutional pde} and the parameters $\theta$ used to generate the tangent bundle are computed.

{\bf Compute the solution ($u^{k} \;\mapsto\; u^{k+1}$):}
The proposed DTB method can be used with any existing time discretization scheme. For simplicity, we take the forward Euler scheme to illustrate the main ideas. Divide the time interval $[0, T]$ into $K$ subintervals with equal stepsize $h=\frac{T}{K}$. Let $u^k(z)=\widehat{u}(\frac{kT}{l}, z)$ represent the approximate solution at time $\frac{kT}{K}$ and $f_{\theta^k}$ a candidate DTB basis. The forward Euler scheme is
\begin{align}
\label{eq: forward euler u level}
    u^{k+1}(z) &= u^k(z) + h F[u^k(z)].
\end{align}
The term $F[u^k(z)]$ is approximated by its optimal DTB approximation $TB(f_{\theta^k}; \alpha^k)$, leading to 
\begin{align}
\label{eq: func update}
    u^{k+1}(z) = u^k(z) + h \cdot TB(f_{\theta^k}; \alpha^k),
\end{align}
where the optimal DTB coefficient $\alpha^k \in \mathbf{R}^m$ is determined by 

\begin{align}
    \alpha^k = G(\theta^k)^{\dagger}P(\theta^k; F[u^k]).
\end{align}
Alternatively, equation \eqref{eq: func update} can be rewritten in terms of the orthogonal projection operator $\mathcal{K}_{\theta^k}$ as
\begin{align}\label{eq: func update kernel form}
    u^{k+1}(z) = u^k + h\mathcal{K}_{\theta^k}[F[u^k]].
\end{align}

{\bf Adapt the parameters ($\theta^{k} \;\mapsto\; \theta^{k+1}$):} It is not necessary to update $\theta^{k}$ according to the same rule as that being used to compute $u^k$. From our experiments, we suggest three options. 

\begin{itemize}
    \item No update for \(\theta\), which works for linear PDEs; see Section \ref{sec: midpoint euler scheme for heat} and Section \ref{sec: heat eq}.
    \item Forward update rule via \(\theta^{k+1} = \theta^{k} + h \alpha^{k}\). This update rule is effective for nonlinear PDEs when the computed \(\alpha^{k}\) is small.
    \item Periodic reset rule, as presented in Proposition \ref{prop: theta update opt}. This approach is recommended when the PDE is nonlinear and the \(\alpha\) dynamics are stiff.
\end{itemize}

\begin{proposition}\label{prop: theta update opt}
    Every \(L\) steps, update \(\theta\) by solving the following optimal approximation, i.e., after solving $u^{jL}$ at $jL$-th step, reset $\theta^{jL}$ as
    \[
      \theta^{jL} \;=\;
      \arg\min_{\tilde\theta}
      \big\|\,u^{jL} - f_{\tilde\theta}\big\|_{L^2},
    \]
    then continue from this new basis parameters.
\end{proposition}

The key insight behind this update rule, observed empirically, is that if \(f_{\theta}\) can sufficiently approximate the target function \(u^{jL}\), then its tangent bundle can likewise approximate the related vector field effectively.

The above procedure demonstrates how the DTB framework can be applied to approximate the solution of evolutionary PDEs by iteratively solving the coefficients $\alpha^k$ at each time step; and adapt the parameter $\theta$ if necessary.
The algorithm is formulated as follows.
\begin{algorithm}[H]
\caption{Forward Euler DTB method for evolutional PDE}
\label{alg: forward euler DTB}
\begin{algorithmic}[1]
\STATE{Set terminal time $t_1$ and number of steps $K$. Set step size $h=t_1/K$.}
\STATE{Initialize $\textrm{DTBset}=\emptyset$ for solved DTB approximation. Initialize solution $u^0=u(0, z)$.}
\FOR{$k=0, \dots, K-1$}
\STATE{Update the parameter $\theta$ if necessary.}
\STATE{
Generate samples $z^j$ from some reference distribution $\lambda$ for $j=1,\dots,N$}.
\STATE{Choose a suitable DTB base $f_{\theta}$, compute the projection vector $p^k=P(\theta; F[u^k])$.}
\STATE{Apply least squares method to solve $\alpha^{k}$ from $G(\theta)\alpha^{k} =p^k$.}
\STATE{Update the parameters of $\theta$ and repeat the above two steps if the projection error is larger than the tolerance.}
\STATE{Denote $\theta^k=\theta$; Update $u^{k+1}(z)=u^{k}(z)+h \cdot \frac{\partial}{\partial\theta}f_{\theta^k}(z) \cdot \alpha^k$.}
\STATE{Update $\textrm{DTBset}+=\{(\theta^k, \alpha^k)\}$.}
\ENDFOR
\STATE{\textbf{Output}: Solution $u^K$ and $\textrm{DTBset}$. }
\end{algorithmic}
\end{algorithm}

When the second option is selected, that is, \(f_\theta\) varies smoothly with respect to \(\theta\) and that \(\alpha^k\) remains bounded, and that the forward update rule \(\theta^{k+1} = \theta^k + h \alpha^k\) is selected, then the following holds.
\begin{proposition}\label{prop: diff between param and func update}
    Suppose the solution \(u^k\) is updated via equation \eqref{eq: func update}, and the parameters are updated as \(\theta^{k+1} = \theta^k + h \alpha^k\). Assume the neural network \(f_\theta\) is smooth in \(\theta\), and \(\partial^2_\theta f_\theta^k(z)\) is uniformly bounded by a positive constant \(C_0\):
    \begin{align}
        |\partial^2_\theta f_{\theta^k}(z)| \leq C_0, \quad \forall z,
    \end{align}
    then, for sufficiently small \(h\) and bounded $|\alpha^k|_{l^2}$, the difference between the DTB solution update \(u^{k+1} - u^k\) and the parameter-based function difference \(f_{\theta^{k+1}} - f_{\theta^k}\) is of the second order in \(h\), i.e.,
    \begin{align}
        f_{\theta^{k+1}} - f_{\theta^k} = (u^{k+1} - u^k) + \mathcal{O}(h^2|\alpha^k|_{l^2}^2C_0).
    \end{align}
    \begin{proof}
        This result follows directly from Taylor's expansion. Its details are omitted. 
    \end{proof}
\end{proposition}
\begin{remark}
\label{rk: extra error in LIT}
In both the LIT-type methods and the DTB method, the vector $\alpha$ is computed by solving equation \eqref{eq: J alpha solution}. The DTB method uses \eqref{eq: func update} to update the solution directly, while the LIT method adapts the parameters, and uses $f_{\theta^k} (z)$ to approximate the solution. In the LIT case, the change in the solution at step $k$ is given by \(f_{\theta^{k+1}} - f_{\theta^k}\). The later approach introduces an additional error term for approximating the right hand side functional $F[u]$:
\begin{align}
\epsilon^{\mathrm{Taylor}} =\frac{1}{h}(f_{\theta^{k+1}} - f_{\theta^k})-\frac{\partial}{\partial\theta}f_{\theta^k}(z) \cdot \alpha^k= \frac{1}{h}\left[(f_{\theta^{k+1}} - f_{\theta^k}) - (u^{k+1} - u^k)\right],    
\end{align}
at each time step. By Proposition \ref{prop: diff between param and func update}, this error is small only when the DNN is smooth, the vector $\alpha^k$ is bounded, and $h$ is small. This may be difficult to satisfy, especially when $F$ is nonlinear and the dimension $d$ is high. 
\end{remark}

One numerical experiment in Section \ref{sec: num verify taylor error} is designed to empirically validate the analysis in Remark \ref{rk: extra error in LIT} and demonstrates the practical impact of the Taylor approximation error $\epsilon^{\mathrm{Taylor}}$, inherent in parameter evolution schemes such as LIT. We compare the performance of DTB and LIT in challenging high-dimensional nonlinear problems where the difference can be significant.

\subsubsection{Implicit and higher order schemes}\label{sec: midpoint euler scheme for heat}
The DTB method can be combined with other traditional time integration schemes, including implicit ones. In this subsection, we illustrate how to build a higher-order, unconditionally stable time integrator using a DTB basis. For clarity, we take the implicit trapezoidal scheme for the heat equation using a fixed DTB basis as an example. 

Consider the following equation with periodic boundary condition.
\begin{align}
\label{eq: heat}
    &\partial_t u(t, z) = \Delta u(t, z),\\
    &u(0, z) = \phi(z).
\end{align}

We apply the implicit trapezoidal rule for the time discretization, 
\begin{align}\label{eq: midpoint euler u level}
    u^{k+1} &= u^k + \frac{h}{2}\left\{\Delta u^k+\Delta u^{k+1}\right\}.
\end{align}
To construct the DTB solver for this scheme, we must choose the optimal coefficients $\alpha^k$ so that the solution is expressed by
\begin{align}\label{eq: midpoint euler u level 1}
    u^{k+1} &= u^k + h \frac{\partial}{\partial\theta}f_{\theta^{k}} \cdot \alpha^{k}.
\end{align}

Since equation \eqref{eq: heat} is linear, we choose the same DTB bases for all time steps, that is, $\theta^k=\theta^0$ for all $k\geq 0$. In this case, we do not adapt the DNN base that generates the DTB. This leads to  
\begin{align}
\label{eq: DTB with same theta}
    u^{k+1} =u^{k} + h\cdot\frac{\partial}{\partial\theta}f_{\theta^{k}} \cdot \alpha^{k}= \phi + h\cdot\frac{\partial}{\partial\theta}f_{\theta^0}(z) \cdot \sum_{i=0}^{k} \alpha^i \eqqcolon  \phi + h\cdot\frac{\partial}{\partial\theta}f_{\theta^0}(z) \cdot s^{k+1}.
\end{align}
Here we introduce the cumulative coefficient $s^k = \sum_{i=0}^{k-1}\alpha^i$ for convenience of notation.

Plugging the DTB expression \eqref{eq: DTB with same theta} in the time discretization rule \eqref{eq: midpoint euler u level} gives 
\begin{align}
\label{eq: heat cn iteration}
    \frac{\partial}{\partial\theta}f_{\theta^0}(z) \cdot s^{k+1} = \frac{\partial}{\partial\theta}f_{\theta^0}(z) \cdot s^k + \Delta \phi + \frac{1}{2}\Delta\frac{\partial}{\partial\theta} f_{\theta^0}(z) \cdot (s^k+s^{k+1}).
\end{align}
Multiplying $\frac{\partial}{\partial\theta}f_{\theta^0}(z)^\top$ from the left to both sides and integrating with respect to $z$, we obtain
\begin{align}
    &\int \frac{\partial}{\partial\theta}f_{\theta^0}(z)^\top\frac{\partial}{\partial\theta}f_{\theta^0}(z) \cdot s^{k+1} dz  \nonumber \\ = & \int (\frac{\partial}{\partial\theta}f_{\theta^0}(z)^\top\frac{\partial}{\partial\theta}f_{\theta^0}(z) \cdot s^k + \Delta \phi + \frac{1}{2}\frac{\partial}{\partial\theta}f_{\theta^0}(z)^\top\Delta\frac{\partial}{\partial\theta} f_{\theta^0}(z) \cdot (s^k+s^{k+1}))dz.
    \label{eq: heat eq update rule 1}
\end{align}
For convenience, we define 
\begin{align}
    & A(\theta^0)=- \int \frac{\partial}{\partial\theta}f_{\theta^0}(z)^\top\Delta\frac{\partial}{\partial\theta} f_{\theta^0}(z)dz=\int \nabla \frac{\partial}{\partial\theta}f_{\theta^0}(z)^\top\nabla \frac{\partial}{\partial\theta}f_{\theta^0}(z)dz, \\
    &b = \int \frac{\partial}{\partial\theta}f_{\theta^0}(z)^\top\Delta \phi\ dz.
\end{align}
The boundary term vanishes when taking the integration by parts because of the periodic boundary conditions. Then equation \eqref{eq: heat eq update rule 1} becomes
\begin{align}
    &G(\theta^0) s^{k+1} = G(\theta^0) \cdot s^k + b - \frac{1}{2}A(\theta^0) (s^k+s^{k+1}).
\end{align}
Solving it for $s^{k+1}$, we have 
\begin{align}
    s^{k+1}&=\left[G(\theta^0) + \frac{1}{2}A(\theta^0)\right]^{\dagger}\left[G(\theta^0) - \frac{1}{2}A(\theta^0)\right]s^{k} + b.
\end{align}
This is the implicit DTB solver for the heat equation. Its corresponding algorithm is given in Algorithm \ref{alg: DTB midpoint Euler for heat}. 
\begin{algorithm}[H]
\caption{Implicit trapezoidal DTB for heat equation}
\label{alg: DTB midpoint Euler for heat}
\begin{algorithmic}[1]
\STATE{Set terminal time $t_1$ and number of steps $K$. Set step size $h=t_1/K$.}
\STATE{Initialize neural network $u_{\theta}$ with parameters $\theta^0$. Initialize solution $\widehat{u}^0=u(0, z)$ and $s^0=\Vec{0}$.}
\STATE{Generate samples $z^j$ from reference distribution $\lambda$ for $j=1,\dots,N$.}
\STATE{Compute the projection vector $b=P(\theta^0; \Delta \phi)=\frac{1}{N}\sum_{i=1}^N\frac{\partial}{\partial\theta}f_{\theta^0}(z^j)^\top\Delta \phi(z^j)$.}
\STATE{Form the empirical matrices $G, A$ as linear operators.}
\FOR{$k=0, \dots, K-1$}
\STATE{Compute $w^k=\left[G(\theta^0) - \frac{1}{2}A(\theta^0)\right]s^{k}$.}
\STATE{Apply least squares method to solve $\alpha$ from $\left[G(\theta^0) + \frac{1}{2}A(\theta^0)\right]\alpha =w^k$.}
\STATE{Denote $s^{k+1}=\alpha + b$; Update $u^{k+1}(z)=\phi(z)+h \cdot \frac{\partial}{\partial\theta}f_{\theta^0}(z) \cdot s^{k+1}$.}
\ENDFOR
\STATE{\textbf{Output}: Solution $u^K$. }
\end{algorithmic}
\end{algorithm}

\subsubsection{Application to Wasserstein flows}\label{sec: WF intro}
The Wasserstein gradient flow (WGF) and Wasserstein Hamiltonian flow (WHF) are PDEs defined on the Wasserstein manifold, which describe the density evolution under certain dynamics. Both can be reformulated as evolutionary equations of push-forward map $T$. In this section, we propose to solve them using the DTB framework. Their numerical examples are presented in Section \ref{sec: wf numerical}.

\textbf{WGF with interaction potential. } We consider the WGF given by
\begin{align}
\label{def: WGF}
    \frac{\partial \rho}{\partial t}= - \textrm{grad}_{W}\mathcal{F}(\rho), \quad \rho(0, z)=\lambda(z),
\end{align}
where \(\text{grad}_W\) denotes the Wasserstein gradient \cite{otto-PME}. 

Given a reference density $\lambda$ and a smooth map $T:\mathbf{R}^d \rightarrow \mathbf{R}^d$, $T$ induces a push-forward distribution with density $T_\sharp \lambda$, defined by
\begin{align*}
    \int_E T_{\sharp}\lambda(x)\,dx = \int_{T^{-1}(E)} \lambda(z) \, dz ,
\end{align*}
for any measurable set \(E \subset \mathbb{R}^d\), where $T^{-1}(E)$ is the pre-image of $E$. 

The WGF has an equivalent formulation through the push-forward map \cite{JIN2025113660}:
\begin{proposition}
\cite{JIN2025113660} Assume $T(t, \cdot):\mathbb{R}^d\rightarrow \mathbb{R}^d$ is smooth for any $t$,  $T(0, \cdot)=Id$, and it solves the following equation,
\begin{align}
    \label{def: operator dyn wgf}
    \dot{ T}(t, z)=-\nabla_X\frac{\delta}{\delta\rho}\mathcal{F}(T(t, \cdot)_{\sharp}\lambda(z), \cdot)\circ T(z) ,
\end{align}    
where $\rho(t)=T(t, \cdot)_{\sharp}\lambda$ is the pushforward density of $\lambda$ through $T(t, \cdot)$. Then $\rho$ solves the WGF \eqref{def: WGF}.
\end{proposition}

If $\mathcal{F}$ is taken as the interaction energy functional
\begin{equation}
\label{eq: interact-F}
\mathcal{F}(\rho) = \frac{1}{2}\iint J(|x-y|)\rho(x)\rho(y)\,dx\,dy,
\end{equation}
where $J(|x-y|)$ is an interaction kernel, \eqref{def: operator dyn wgf} becomes
\begin{equation}\label{eq: wgf T interaction}
\frac{d}{dt}T(t, z) = - \nabla (J * \rho)\circ T(t, z), \quad T(0, \cdot)=Id,
\end{equation}    
Or equivalently, for a particle $X(t)=T(t, z)$, we have 
\begin{equation}\label{eq: wgf T interaction X level}
\frac{d}{dt}X(t) = - \nabla (J * \rho)(X(t)), \quad X(0)=z,
\end{equation}    
$J*\rho(x)=\int J(|x-y|)\rho(y)dy$ denotes the convolution of $J$ with $\rho$. 

We apply the forward Euler scheme given in Algorithm~\ref{alg: forward euler DTB} to update the approximate pushforward map $X^{k}=T(\{\theta^i\}_{i=1}^k; Z)$ iteratively,
\begin{align}\label{eq: wgf particle iteration}
    X^{k+1} =X^{k} + h\mathcal{K}_{\theta;\rho^k}[-\nabla(J * \rho^k)](X^{k}),\quad X^{0}=Z.
\end{align}
where $Z$ is the random variable with density $\lambda$, $\rho^k$ is the density of $X^k$, and $\mathcal{K}_{\theta;\rho}$ is the projection operator on the tangent bundle $\mathcal{T}_{\theta}$. The evaluation of $\mathcal{K}_{\theta;\rho}$ can be approximated by the Monte Carlo samples which are detailed in  
Appendix~\ref{appendix: wgf_details}.

\textbf{WHF with a linear potential. }The WHF \cite{chow2020wasserstein} is given by
\begin{subequations}
\label{eq:WHF-in-dual-coordinates}
\begin{align}
\label{eq: whf op}
    &\partial_t\rho = \frac{\delta}{\delta\Phi}\mathcal{H}(\rho, \Phi),\\
    &\partial_t\Phi = -\frac{\delta}{\delta\rho}\mathcal{H}(\rho, \Phi),
\end{align}
\end{subequations}
where the Hamiltonian \(\mathcal{H}\) is defined as
\begin{align}
\label{eq:WHF-H}
\mathcal{H}(\rho, \Phi) = \int_{\mathbb{R}^d} \frac{1}{2}|\nabla \Phi(z)|^2\rho(z) dx+\mathcal{F}(\rho).
\end{align}
As shown in \cite{doi:10.1137/23M159281X}, the above system can also be reformulated in terms of the push-forward map.
\begin{proposition}
\cite{doi:10.1137/23M159281X} Assume $T(t, \cdot):\mathbb{R}^d\rightarrow \mathbb{R}^d$ is smooth for any $t$, $T(0, \cdot)=Id$, and it solves the following equation,
\begin{align}
    \label{def: operator dyn whf}
    \frac{d^2}{dt^2}T(t, z)=-\nabla_X\frac{\delta}{\delta\rho}\mathcal{F}(T(t, \cdot)_{\sharp}\lambda(z), \cdot)\circ T(z) ,
\end{align}    
where $\rho(t)=T(t, \cdot)_{\sharp}\lambda$ is the pushforward density of $\lambda$ through $T(t, \cdot)$. Then $\rho$ solves the WHF \eqref{eq:WHF-in-dual-coordinates}.
\end{proposition}


Consider a special case where $ \mathcal{F}(\rho) = \int V(z)\rho(z)dx$, \eqref{def: operator dyn whf} can be simplified to:
\begin{align}
\label{eq: whf linear op}
    \frac{d^2}{dt^2}T(t,z)=-\nabla V(\cdot)\circ T(z).
\end{align}

Similarly to the first-order case in \eqref{eq: wgf particle iteration}, we solve the second-order system \eqref{eq: whf linear op} by first converting it to a coupled system for position $X$ and velocity $\Lambda = \dot{X}$. The acceleration field $-\nabla V(\cdot)$ is approximated at each step by projecting it on the DNN tangent space $\mathcal{T}_{\theta}$ using Algorithm~\ref{alg: forward euler DTB}. The numerical update at step $k$ is:
\begin{align}
    &X^{k+1} =X^{k} + \frac{h}{2}(\Lambda^k+\Lambda^{k+1}),\quad X^{0}=Z.\\
    &\Lambda^{k+1}= \Lambda^{k} + h\mathcal{K}_{\theta;\rho^k}[-\nabla V(\cdot)](X^{k}),\quad \Lambda ^{0}=\nabla \Phi(0, Z).
\end{align}

\subsection{Theoretical error estimation} \label{sec: error estimate}
From error analysis viewpoint, DTB behaves similarly to traditional numerical schemes with one key difference: the spatial discretization error is replaced by the error arising from the function approximation step. To be more specific, assume $S$ is a \textbf{time discretization rule} and $\hat{v}$ is the corresponding numerical solution, i.e., $\hat{v}^{k+1}=\hat{v}^k + hS\hat{v}^k$. For example, the forward Euler scheme corresponds to 
\begin{align}
    S_1=F[\cdot] = \Delta,
\end{align}
and the implicit trapezoidal scheme \eqref{eq: midpoint euler u level} for heat equation 
corresponds to:
\begin{align}
    S_2 = \frac{1}{h}(I-\frac{h}{2}\Delta)^\dagger(I+\frac{h}{2}\Delta) - \frac{1}{h}I,
\end{align}
where $\dagger$ stands for the pseudo inverse of the operator. Notice that $S$ is an operator in the function space \textbf{without} space discretization. 

Let \(\hat{u}\) represent the DTB solution constructed with the same time discretization rule \(S\), i.e., the DTB solution at the next step is computed from $\hat{u}^{k+1} = \hat{u}^k + h\mathcal{K}_{\theta^k}[S\hat{u}^k]$ where $\mathcal{K}_{\theta^k}$ is the projection operator on the DTB basis as defined in \eqref{eq: func proj op def}. The initial values for both $\hat{v}$ and $\hat{u}$ are set to the initial condition $u(0, \cdot)$,
\begin{align}\label{eq: num scheme initial}
    \hat{v}^0=\hat{u}^0=u(0, \cdot)=\phi.
\end{align}
Then the error in the DTB solution can be estimated as given in the following theorem.

\begin{theorem}
    Assume the numerical iteration rule $S$ is linear and satisfy the following inequality
    \begin{align}
        \|(I+hS)w\|_{L^2} \leq (1+hC_1)\|w\|_{L^2},
    \end{align}
    where $C_1>0$ is a constant. Denote $\delta^k=\|\mathcal{K}_{\theta^k}[S\hat{u}^k]-S\hat{u}^k\|_{L^2}$ as the error in solving the linear system at step $k$, then the error $\epsilon^{DN}(k)=\|\hat{u}^k-\hat{v}^k\|_{L^2}$ between DTB solution $\hat{u}$ and numerical solution $\hat{v}$ satisfies the inequality
        \begin{align}
        \epsilon^{DN}(k+1)
        &\leq (1+hC_1)\epsilon ^{DN}(k) + \delta^k.
        \end{align}
    Furthermore, if $\delta^k$ can be bounded from above by $\delta_{1}$, then the error from DTB solution, $\epsilon^D(k) = \|\hat{u}^k-u(kh)\|_{L^2}$, can be upper bounded by
    \begin{align}\label{ineq: DTB scheme error}
        \epsilon^D(k) &\leq \epsilon^S(k) + \frac{\delta_1}{C_1h} \left[e^{C_1hk} -1\right].
    \end{align}
    where $\epsilon ^S(k)=\|\hat{v}^k-u(kh)\|_{L^2}$ is the solution error from the numerical scheme, and \(u(t)\) is the exact solution to PDE \eqref{eq: evolutional pde}.

    \begin{proof}
    Notice that $\hat{u}^{k+1}=(I+hS)\hat{u}^k + h(\mathcal{K}_{\theta^k}[S\hat{u}^k]-S\hat{u}^k)$ and $\hat{v}^{k+1}=(I+hS)\hat{v}^k$. By triangle inequality we have:
    \begin{align}
        \epsilon^{DN}(k+1) &= \|(I+hS)\hat{u}^k + h(\mathcal{K}_{\theta^k}[S\hat{u}^k]-S\hat{u}^k)-(I+hS)\hat{v}^k\|_{L^2} \nonumber\\
        &=\|(I+hS)(\hat{u}^k-\hat{v}^k)+ h(\mathcal{K}_{\theta^k}[S\hat{u}^k]-S\hat{u}^k)\|_{L^2}\nonumber\\
        &\leq \|(I+hS)(\hat{u}^k-\hat{v}^k)\|_{L^2}+ h\|\mathcal{K}_{\theta^k}[S\hat{u}^k]-S\hat{u}^k\|_{L^2}\nonumber\\
        &\leq (1+hC_1)\epsilon ^{DN}(k) + \delta^k.
    \end{align}
        Given $\delta^k\leq \delta_1$ for all $k$, from the Grownwall inequality we get
        \begin{align}
        \epsilon^{DN}(k) &\leq (\epsilon^{\textrm{DN}}(0)+\frac{\delta_1}{C_1h} )e^{C_1hk} -\frac{\delta_1}{C_1h} =\frac{\delta_1}{C_1h} \left[e^{C_1hk} -1\right].
        \end{align}
        
        The last equality comes from the initial condition \eqref{eq: num scheme initial}. By triangle inequality, we proved \eqref{ineq: DTB scheme error}.
    \end{proof}
    \end{theorem}

The proof of the theorem reveals that the total error in the DTB framework can be systematically decomposed into three components. 
    \begin{enumerate}
      \item \emph{Sampling Error}: This error originates from the random sampling to construct the linear system. Increasing the number of samples may systematically reduce this error.
      \item \emph{Local Approximation Error}: This component arises from projecting the PDE's RHS onto the finite-dimensional tangent bundle. Enriching the basis set, periodically re-selecting the DTB base and solving the linear system with higher accuracy can mitigate this error.
      \item \emph{Numerical Integration Error}: This error is associated with the numerical scheme employed to integrate the solved approximation to $F[u]$. The choice of scheme order and discretization parameters directly influence this component. Employing higher-order or adaptive time-integration schemes can effectively control this error.
    \end{enumerate}
These sources of error are controllable and can be systematically reduced through appropriate strategies such as increasing the sample size, enriching the basis, and adopting higher-order numerical schemes. However, the precise relationship between the error and the sample size or the richness of the basis is theoretically still unclear. They are important questions that deserve further investigation.  

To conclude this section, we highlight the key advantages of the proposed method.
    \begin{itemize}
      \item \emph{No nonconvex training}: The DTB approach replaces the nonconvex training procedures with the solution of linear systems, simplifying the implementation and reducing potential issues related to local minima created by the DNN.
      \item \emph{Basis management for improved conditioning}: The framework allows for flexible updates or resets of the neural network parameters, thereby preventing basis degeneration and maintaining numerical stability throughout the evolution.
      \item \emph{Exploitation of local linear structure}: Leveraging the local linearity of the tangent bundle ensures robust and efficient updates in the function space, enhancing the performance in  stability and convergence.
      \item \emph{Mesh-free, high-dimensional scalability}: The method is inherently mesh-free and well-suited for high-dimensional problems, making it applicable to complex, high-dimensional PDEs without the need for grid discretization.
    \end{itemize}

By combining the mesh-free structure of DNN with the robustness and stability of classical linear system solvers, the DTB offers an effective alternative to solve high-dimensional PDEs.

\section{Numerical Results}
\label{sec: numerical method}
In this section, we assess the performance of the proposed DTB method through a series of numerical experiments. These examples include the approximation of the high-dimensional function, the demonstration of the proposition \ref{prop: diff between param and func update}, the heat equation, the Allen–Cahn equation, and the WGF and WHF.

\subsection{Function approximation}\label{sec: num func approx}
We begin by evaluating the DTB method on the task of approximating
$\mathcal{O}w$, where $w\!:\![-1,1]^5\!\to\!\mathbf{R}$ is a nonlinear
function and $\mathcal{O}$ is a prescribed operator.

In the experiment, we choose:
\begin{align}\label{eq: AC u0}
    w(z)&=-1+2\frac{\tilde{w}(z)+6}{13} ,
\end{align}
with
\begin{align}
    \tilde{w}(z) &=cz_1^2 + sz_2^3 + 1.5*sz_1^2*cz_5+3*\frac{1-e^{sz_2}}{1+e^{cz_4}}+2*sz_1*cz_3 \nonumber\\
    &+\log(2+cz_4*sz_1^2) * \frac{1}{e^{cz_5+0.3*sz_4}} + 3 * \log(3+cz_2+sz_5)*\frac{1}{3+sz_3},
\end{align}
where $z=(z_i)_{i=1}^5\in [-1, 1]^5$, $cz_i=\cos (\pi z_i), sz_i=\sin (\pi z_i)$.
We test three different operators:
\begin{align}
    &\mathcal{O}_1 w = 0.005*\Delta w + w - w^3,\\
    &\mathcal{O}_2 w = 0.02* \Delta w,\\
    &\mathcal{O}_3 w = w^4.
\end{align}

The base network $f_\theta=\phi_{\theta_1}\circ PL_{\theta_2}$ is a composition of Multi-component and Multi-layer Neural Networks (MMNN) $\phi_{\theta_1}$ of size (366, 25, 7)\cite{zhang2025structuredbalancedmulticomponentmultilayer}, and a periodic embedding layer $PL_{\theta_2}: [-1, 1]^5\rightarrow \mathbf{R}^{200}$ defined by
\begin{align}
    PL_{\theta_2}(z)=[\cos(\pi z_i + \psi_{i, j})]_{1\leq i\leq 5, 1\leq j\leq 40}, \quad \theta_2=\{\psi_{i, j}\}_{1\leq i\leq 5, 1\leq j\leq 40}. 
\end{align}
We use the $\cos$ function to enforce the periodic boundary condition and $\psi_{i, j}$ serves as a shift for the input variable. The total number of parameters in $PL_{\theta_2}$ is $200=5\times 40$, which also corresponds to the dimension of the output variable of $PL_{\theta_2}$.

We pre-train $f_\theta$ to minimize $\lVert f_\theta - w\rVert_{L^2}$ using the training setup in \cite{zhang2025structuredbalancedmulticomponentmultilayer} with $1000$ iterations and the parameters in $PL_{\theta_2}$ is fixed during the process. Algorithm \ref{alg: J-DTB} is used to solve the linear system for each operator~$\mathcal O_i$, with $20000$ samples and a subspace of the DTB $\mathcal{B}(f, \theta)$ with $6000$ randomly selected vectors are used to generate the linear system
\begin{align}\label{eq: alpha for approx Ow}
    \frac{\partial}{\partial\theta}f_{\theta}(z) \cdot \alpha = \mathcal{O}_i w(z).
\end{align}
To illustrate the performance, we visualize the approximation on five two-dimensional hyperplanes in
$[-1,1]^5$. They are defined by

\begin{enumerate}[label=(\alph*)]
    \item $-z_1=z_2, z_4=\frac{1}{2}(z_2-z_5), z_3=0$,
    \item $z_5=\frac{1}{2}(z_1+z_3), z_2=z_4=0$,
    \item $z_1=z_2=0.3, z_5=0.15-\frac{1}{2}z_3$,
    \item $z_1=0.4z_4+0.6z_5, z_2=z_3=0.8$,
    \item $z_2=0.75z_1+0.25z_5, z_4=0.25z_1+0.75z_5, z_3=0.5$.
\end{enumerate}

\begin{figure}[htbp]
\centering
\begin{tabular}{c}
\vspace{-18pt}  
\includegraphics[width=.95\linewidth]{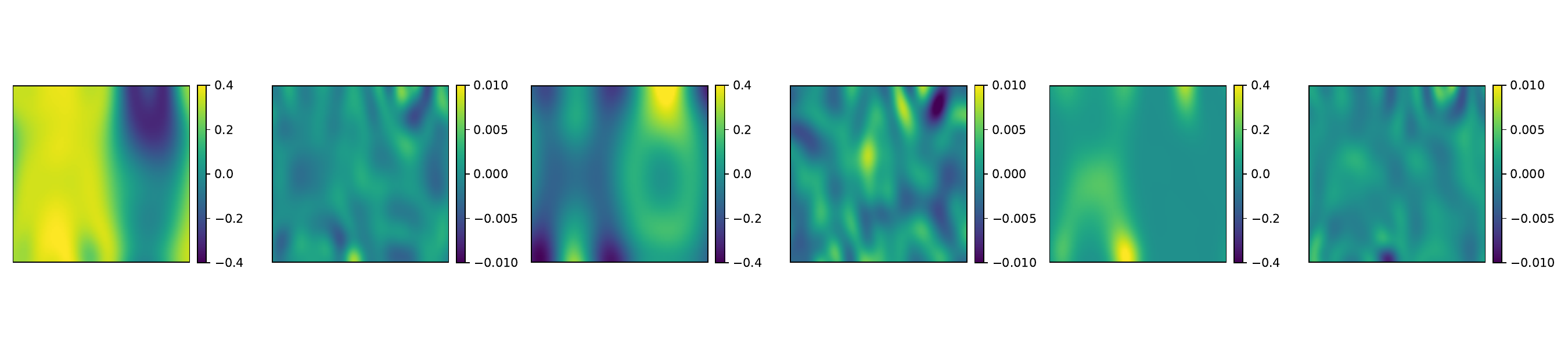}\\
\vspace{-18pt}  
\includegraphics[width=.95\linewidth]{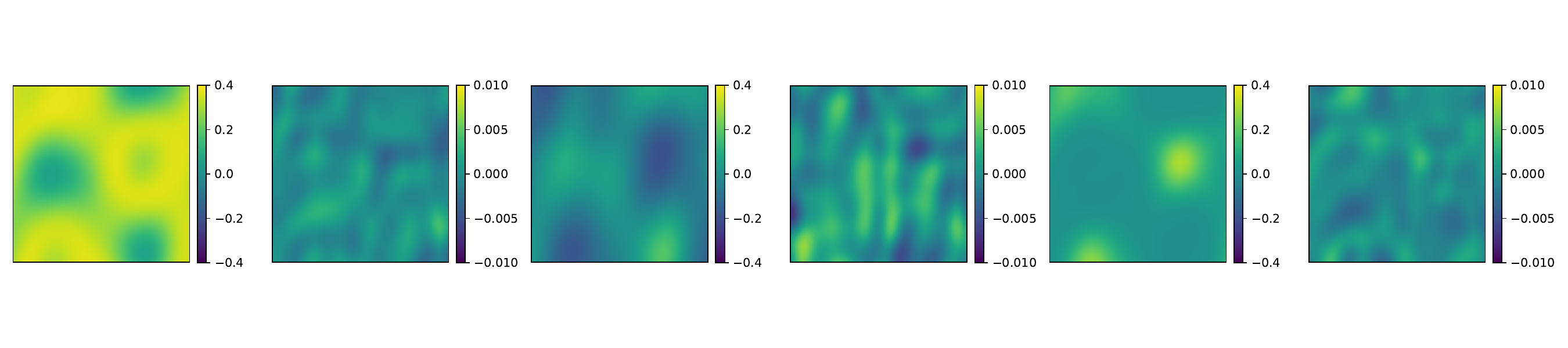}\\
\vspace{-18pt}  
\includegraphics[width=.95\linewidth]{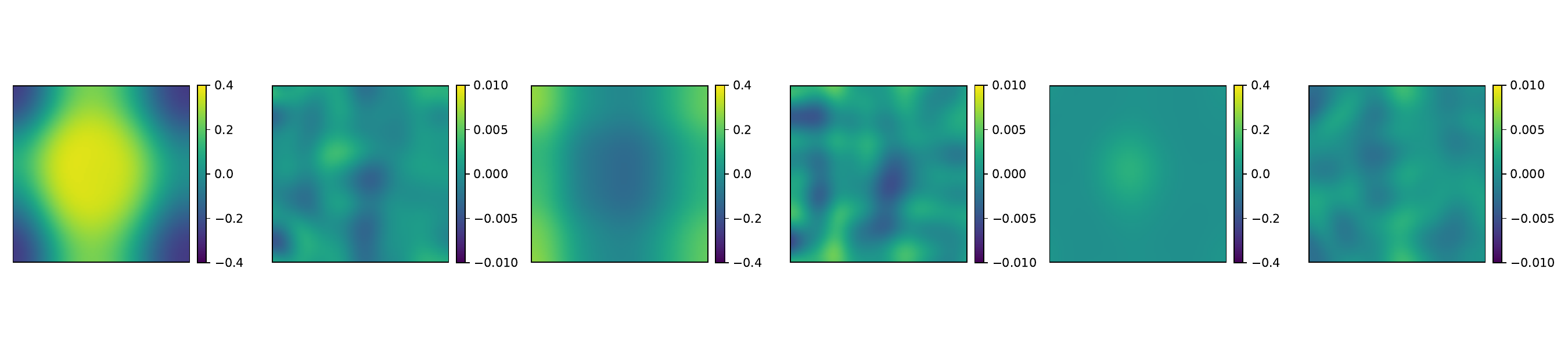}\\
\vspace{-18pt}  
\includegraphics[width=.95\linewidth]{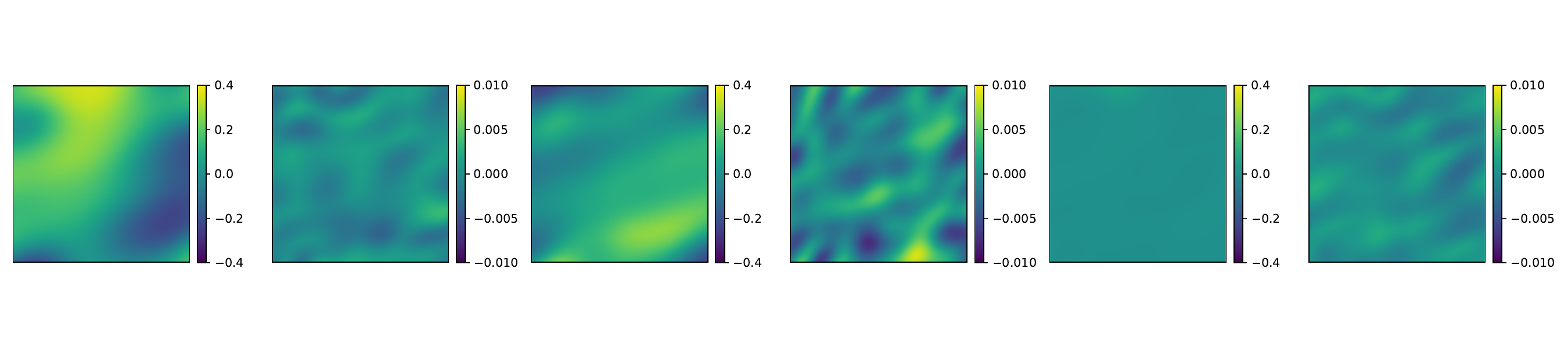}\\
\vspace{-18pt}  
\includegraphics[width=.95\linewidth]{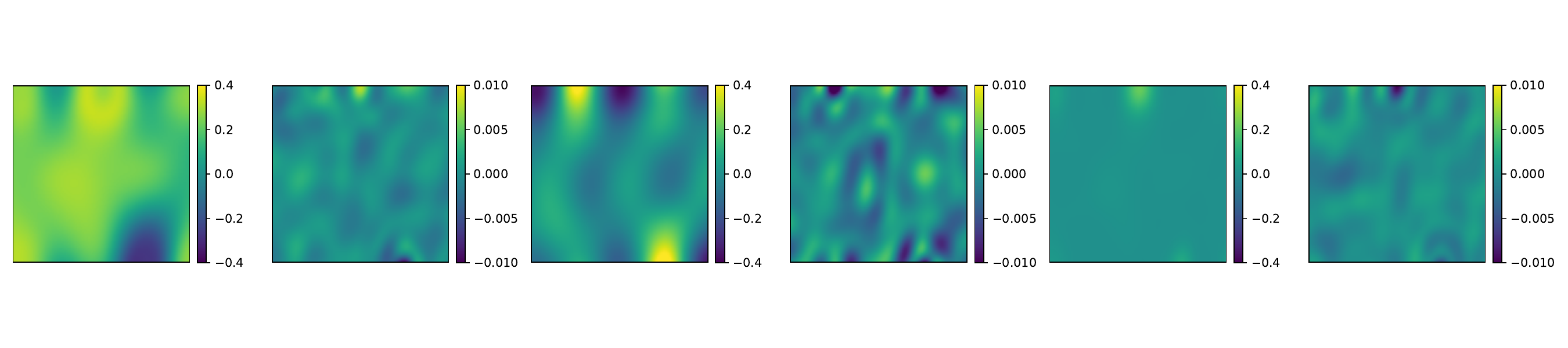}\\
\vspace{20pt}
Approx to $\mathcal{O}_1w$\qquad Error \qquad \quad \ \ Approx to $\mathcal{O}_2w$\qquad Error\qquad \quad \ \ Approx to $\mathcal{O}_3w$\qquad Error\qquad \qquad \quad\\
The rows are projections onto subplane (a)-(e)
\end{tabular}
\caption{DTB approximation to 5-D function generated from three different operators $\{\mathcal{O}_i:i=1, 2, 3\}$. Evaluation is performed on the five two-dimensional hyperplanes (a)-(e). DTB accurately reproduces the structures of $\mathcal{O}_i w$ with pretrained base parameters, yielding small errors.}
\label{fig: func approx 5d}
\end{figure}

Figure~\ref{fig: func approx 5d} displays the cross sections of the DTB approximation of the nonlinear target functions $\mathcal{O}_iw$, along with their corresponding absolute errors. The results demonstrate good accuracy, with the absolute error consistently on the order of \(10^{-3}\) (smaller than $0.01$) in the domain. This accuracy is particularly noteworthy given the combined challenges of high dimensionality and function complexity. With \(20,000\) points sampled in a 5-D hypercube, the data are relatively sparse, and the expected distance between two neighboring points is approximately $0.06$. This indicates that the DTB must accurately capture the complex, nonlinear behavior while generalizing over distances that are of one magnitude larger than the error size. Furthermore, we observe that the results are robust. Computing the model with different random samples and basis initializations yields qualitatively similar performance and error magnitudes, indicating that the approximation is not sensitive to the random choices in the computation. 

\subsection{A numerical experiment demonstrating  Proposition \ref{prop: diff between param and func update}} \label{sec: num verify taylor error}
Proposition \ref{prop: diff between param and func update} discusses the difference between the update of the DTB solution and that of the parameter-based solution. This difference remains small only when the time stepsize $h$ is sufficiently small, the update vector \(\alpha\) remains bounded, and the neural network \(f_\theta\) is smooth enough. However, in our experiments, these conditions are often difficult to satisfy simultaneously. In particular, when solving the linear system with high accuracy, the resulting vector \(\alpha\) tends to have a large norm. We show some experiments to illustrate these points.

We follow the same experimental setup as in the previous section, except that we set the operator $\mathcal{O}$ to be:
\begin{align}
    &\mathcal{O} w = 0.01*\Delta w + w - w^3,
\end{align}
After pre-training $f_{\theta}$ by minimizing $\lVert f_\theta - w\rVert_{L^2}$, we vary the hyperparameter \textbf{rcond} in the JAX linear system solver from $10^{-3}$ to $10^{-7}$, and solve for $\alpha$ using algorithm \ref{alg: J-DTB}. We would like to mention that rcond corresponds to the tolerance used as the stopping criterion in the linear solver. Please see the documentation \cite{jax2018github} of \textbf{jax.numpy.linalg.lstsq} function for the details.   

We record the following metrics in table \ref{tab: approx err from taylor}.
\begin{itemize}
    \item $l^2$ norm of solved $\alpha $ as vector in $\mathbf{R}^{6000}$.
    \item the relative error from linear system solver, defined as:
    \begin{align}
        \epsilon^{\textrm{rel\_LS}} = \frac{1}{\|\mathcal{O}w\|_{L^2}}\|\frac{\partial}{\partial\theta}f_{\theta}(z) \cdot \alpha - \mathcal{O}w(z)\|_{L^2}=\frac{1}{\|\mathcal{O}w\|_{L^2}}\|(\mathcal{K}_{\theta}-I) [\mathcal{O}w]\|_{L^2}.
    \end{align}
    Notice that the relative approximation error from DTB solution is identical to $\epsilon^{\textrm{rel\_LS}}$ by definition \eqref{eq: func opt approx}.
    \item the relative error $\epsilon^{\textrm{rel\_T}}$, which measures the deviation of the parameter-based function difference $\frac{1}{h}(f_{\theta +h\alpha}(z)-f_{\theta})$ towards the target function $\mathcal{O}w(z)$:
    \begin{align}\label{eq: relative epsilon taylor}
        \epsilon ^{\textrm{rel\_T}}=\frac{1}{\|\mathcal{O}w\|_{L^2}}\|\frac{1}{h}\left[f_{\theta +h\alpha}(z)-f_{\theta}\right] - \mathcal{O}w(z)\|_{L^2}.
    \end{align}
\end{itemize}

\begin{table}[ht]
\centering
\begin{tabular}{|c|c|c|c|c|c|c|c|c|}
\hline
\multirow{2}{*}{rcond} & \multirow{2}{*}{$\epsilon^{\textrm{rel\_LS}}$} & \multirow{2}{*}{$|\alpha|_{l^2}$} & \multicolumn{6}{c|}{empirical $\epsilon ^{\textrm{rel\_T}}$ for different step size $h$} \\
\cline{4-9}
 & & & 0.001 & 0.002 & 0.005 & 0.01 & 0.02 & 0.03 \\
\hline
1e-3 & 0.0376 & 15.04 & 0.0384 & 0.0411 & 0.0565 & 0.0930 & 0.1762 & 0.2659 \\
\hline
1e-4 & 0.0127 & 25.81 & 0.0265 & 0.0485 & 0.1198 & 0.2456 & 0.5229 & 0.8405 \\
\hline
1e-05 & 0.0091 & 109.39 & 0.0974 & 0.1951 & 0.4942 & 1.0108 & 2.1096 & 3.2969 \\
\hline
1e-06 & 0.0069 & 602.69 & 5.5993 & 11.1376 & 26.5500 & 43.7735 & 40.9535 & 28.8702 \\
\hline
1e-07 & 0.0060 & 4365.81 & 103.2878 & 212.7789 & 462.6030 & 403.7084 & 682.9411 & 745.1676 \\
\hline
\end{tabular}
\caption{Results of the experiments varying the linear solver's regularization parameter ($\textrm{rcond}$). The table shows the relative solver error $\epsilon^{\mathrm{rel\_LS}}$, the $l^2$ norm of the solution vector $|\alpha|_{l^2}$, and the empirical relative errors $\epsilon^{\mathrm{rel\_T}}$ for different step sizes $h$.}
\label{tab: approx err from taylor}
\end{table}

The results reveal a critical limitation of the LIT approach for such problems. We observe that as the tolerance rcond decreases, the computed $|\alpha|_{l^2}$ increases. The magnitude of $\epsilon^{\mathrm{rel\_T}}$ increases significantly, leading to inaccurate or even incorrect approximations of the solution regardless of the step size $h$. In contrast, the error produced by the proposed DTB method $\epsilon^{\mathrm{rel\_{LS}}}$ decreases as rcond decreases. This shows that the DTB method can maintain high accuracy and stability throughout the simulation. These findings highlight the fundamental differences between LIT-type methods and the DTB framework, especially their capability of handling complex dynamical solutions. The observations confirm our discussion in Remark \ref{rk: extra error in LIT}.

\subsection{Evolutionary PDE}\label{sec:evolution-pdes}
We next demonstrate the performance of the DTB solver on evolutionary PDEs.

\subsubsection{Heat equation}\label{sec: heat eq}
Consider the heat equation
\begin{equation}\label{eq:heat}
\partial_t u(t, z) = 0.1*\Delta u(t, z), \qquad z\in [-1,1]^2,\; t\in[0,4],
\end{equation}
with periodic boundary condition. The initial condition is set to be 
\begin{align}
\label{eq: 2d heat initial func}
    u(t, z)=u_0(z_1, z_2) = \frac{1}{100}\left[\exp\left( 3s_1+s_2\right)+\exp\left(- 3s_1+s_2\right)-\exp\left( 3s_1-s_2\right)-\exp\left( -3s_1-s_2\right)\right],
\end{align}
where $z=(z_1, z_2)$ and $ s_i=\sin z_i, i=1, 2$.

We discretize \eqref{eq:heat} in time by the implicit trapezoidal DTB scheme
(algorithm~\ref{alg: DTB midpoint Euler for heat}) with time step $h=0.01$, the DTB base neural network is a periodic MLP with 7 hidden layers and 40 neurons in each layer,
initialized as in \cite{chen2024teng}, and $6000$ random selected base vectors from $\mathcal{B}(f, \theta)$ are used for calculation. Since the heat equation is linear, we keep the DTB base parameters $\theta^k$ \textbf{fixed} for all steps, that is, $\theta^k=\theta^0, k\geq 1$.
Fig \ref{fig: heat eq 2d} shows that the DTB method achieves high accuracy in solving the 2-D heat equation compared to the reference solution obtained by the Fourier spectral method \cite{chen2024teng}.

\begin{figure}[htbp]
\centering
\begin{tabular}{c}
\vspace{-18pt}  
\includegraphics[width=.95\linewidth]{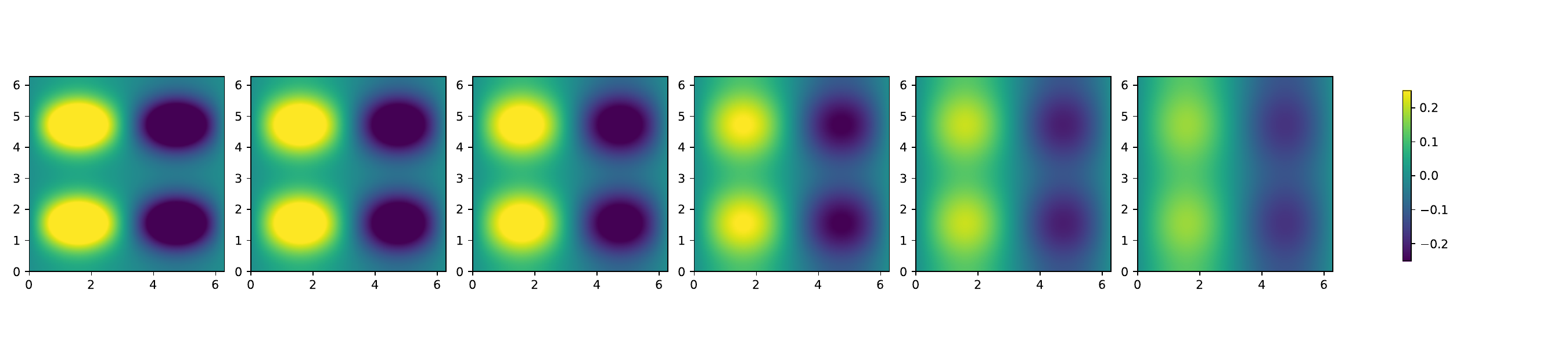}\\
\vspace{-18pt}  
\includegraphics[width=.95\linewidth]{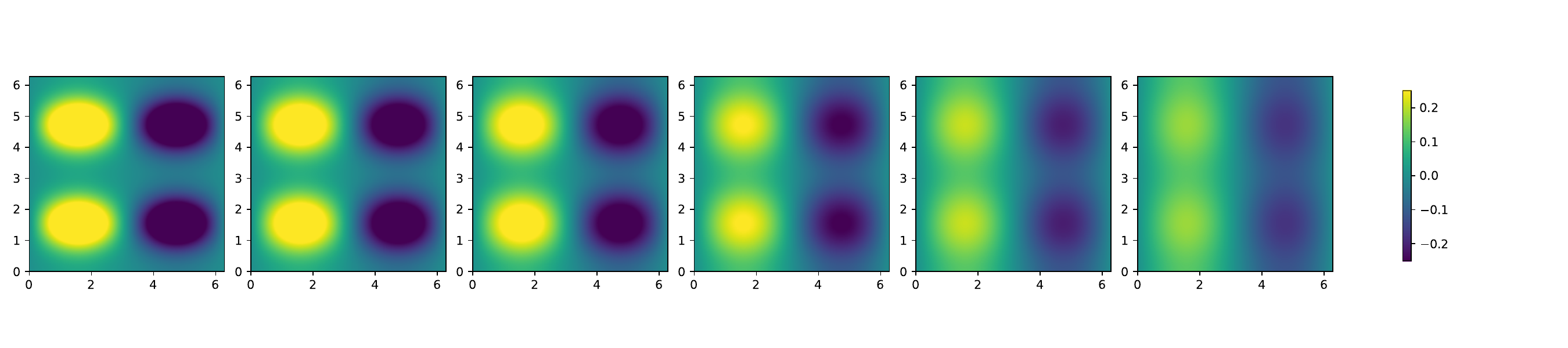} \\
\vspace{-18pt}  
\includegraphics[width=.95\linewidth]{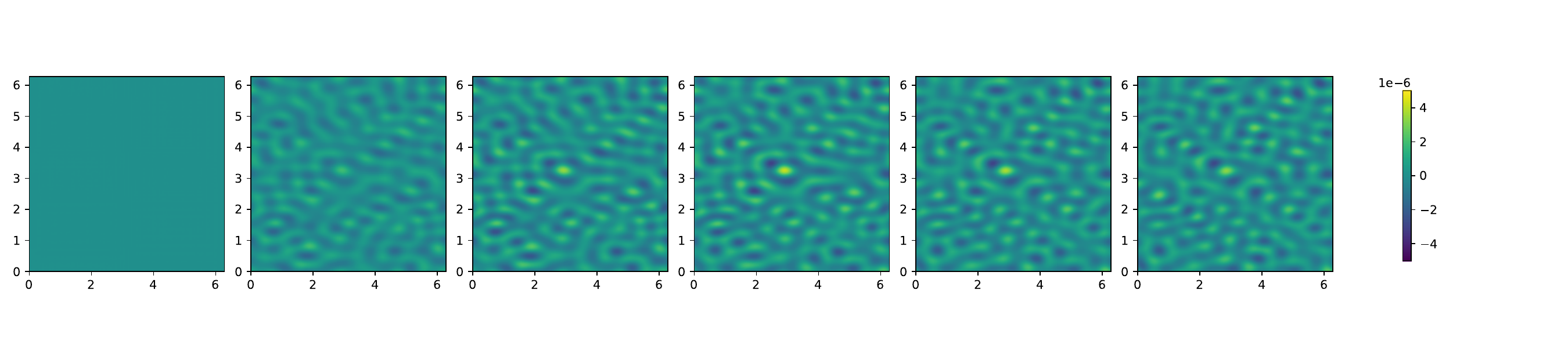}\\
\vspace{20pt}
t=0.\qquad \qquad t=0.5\qquad\qquad t=1.0\qquad\qquad t=2.0\qquad\qquad t=3.0\qquad\qquad t=4.0\qquad \qquad \quad .\\
Upper: DTB solution; middle: reference solution; lower: solution error
\end{tabular}
\caption{Solution of the 2‑D heat equation \eqref{eq: heat} computed with the implicit trapezoidal DTB scheme (Algorithm~\ref{alg: DTB midpoint Euler for heat}). The  reference solution is obtained via a Fourier spectral method as in \cite{chen2024teng}. The DTB solution closely matches the spectral reference, exhibiting smooth diffusion and small errors across all times.}
\label{fig: heat eq 2d}
\end{figure}

\begin{remark}
The spatial approximation error with the pre-trained parameters is
$\mathcal O(10^{-6})$, whereas the local time discretization error of the first-order scheme is $\mathcal O(h)\approx 10^{-2}$. Compared to the forward Euler-based DTB solver as in Fig \ref{fig: heat forward euler 2d}, the implicit trapezoidal DTB method, which is second order in time, reduces the error substantially.
\end{remark}

\begin{figure}[htbp]
\centering
\begin{tabular}{c}
\vspace{-18pt}  
\includegraphics[width=.95\linewidth]{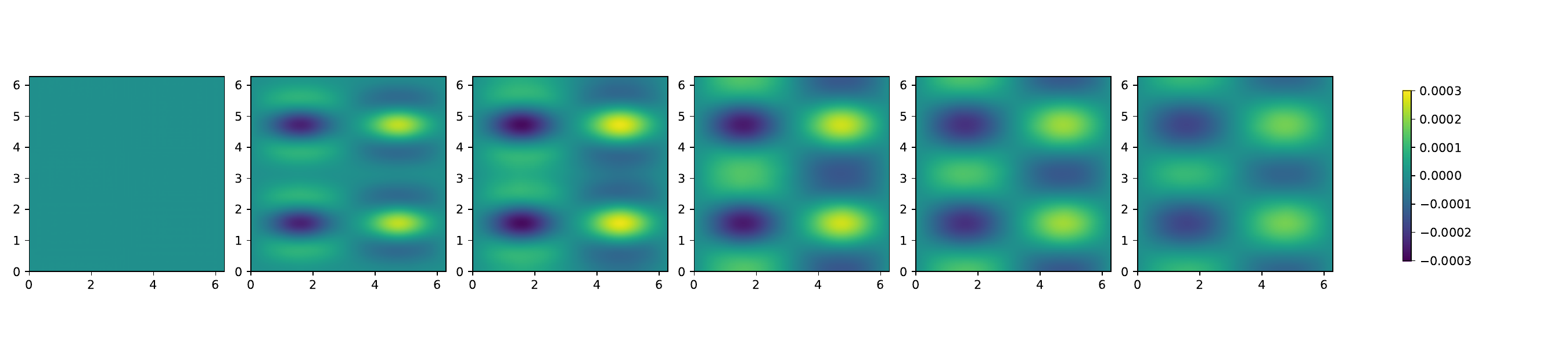}\\
\vspace{20pt}
t=0.\qquad \qquad t=0.5\qquad\qquad t=1.0\qquad\qquad t=2.0\qquad\qquad t=3.0\qquad\qquad t=4.0\qquad \qquad \quad .\\
\end{tabular}
\caption{2‑D heat equation \eqref{eq: heat} solved with the forward‑Euler DTB scheme Algorithm~\ref{alg: forward euler DTB}. The setup matches Figure \ref{fig: heat eq 2d} except the time integrator. Compared with the implicit trapezoidal scheme in Figure \ref{fig: heat eq 2d}, the forward‑Euler DTB exhibits noticeably larger errors.}
\label{fig: heat forward euler 2d}
\end{figure}

\subsubsection{Allen Cahn equation} We consider the Allen-Cahn equation in two and five spatial dimensions, respectively, that is 
\begin{align}\label{eq: AC}
    \partial_t u(t, z)=\nu\Delta u(t, z) + u(t, z) -u(t, z)^3,\quad z\in [-1, 1]^d,\ t\in[0, T],
\end{align}
with periodic boundary condition and $d=2, 5$.
\subsubsection{2-D case} We take $\nu=0.005,\ T=4$ and consider the same initial condition \eqref{eq: 2d heat initial func} as in section \ref{sec: heat eq}.
    We discretize \eqref{eq: AC} in time with time step $h = 0.01$. 
    The DTB architecture and initialization are identical to that used for the heat
equation. Since \eqref{eq: AC} is a nonlinear PDE, we \textbf{update} the DTB parameters $\theta$ at each time step, with the following update rule:
\begin{align}\label{eq: forward theta update}
    \theta^{k+1}=\theta^k + h\alpha^k.
\end{align}
We applied an extra correction step when solving for $\alpha^k$ by the Runge-Kutta scheme; see Appendix \ref{sec: ac 2d alg} for the details of the algorithm design.
We compare the numerical results with those obtained by the spectrum method in Fig. \ref{fig: AC sol 2d}. The error is uniformly small throughout the simulation.
\begin{figure}[htbp]
\centering
\begin{tabular}{c}
\vspace{-18pt}  
\includegraphics[width=.95\linewidth]{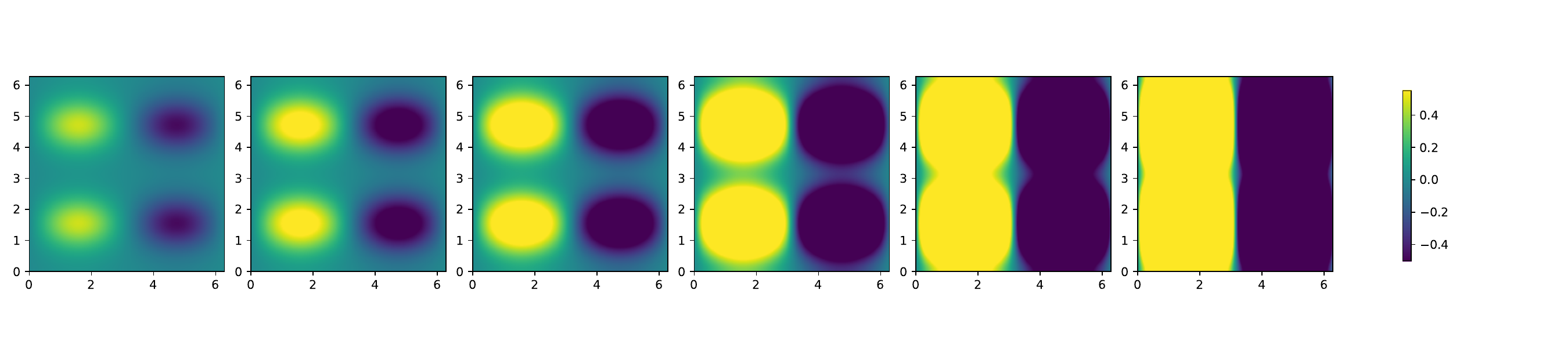}\\
\vspace{-18pt}  
\includegraphics[width=.95\linewidth]{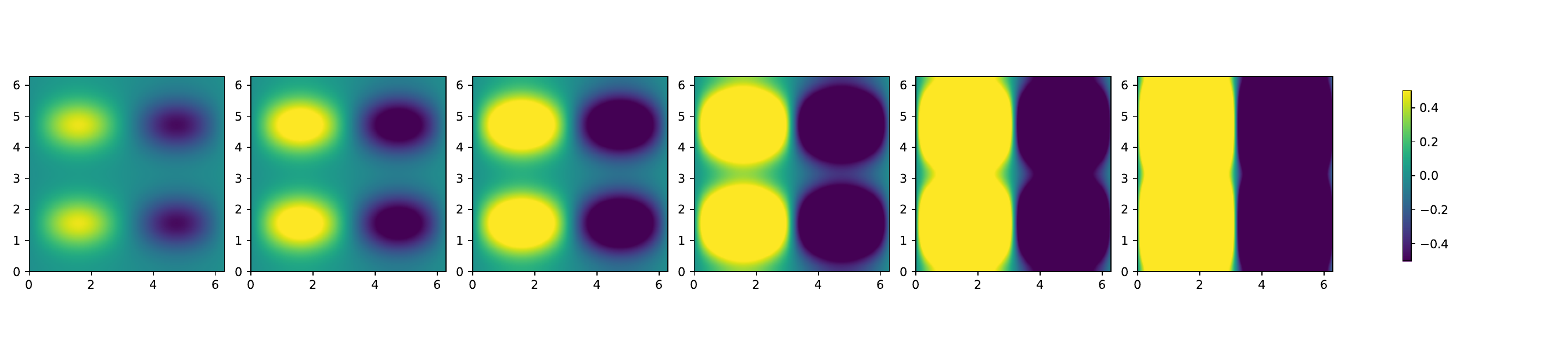} \\
\vspace{-18pt}  
\includegraphics[width=.95\linewidth]{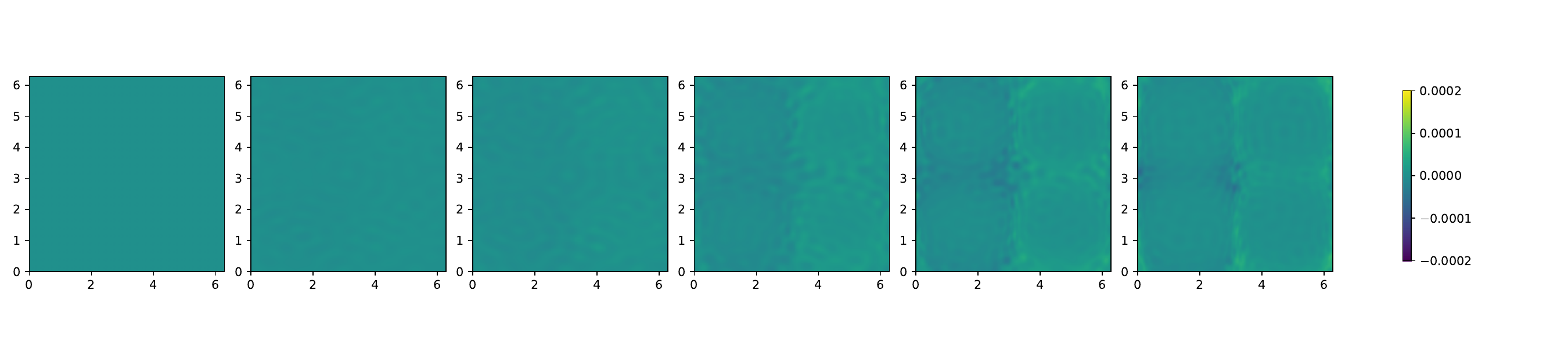}\\
\vspace{20pt}
t=0.\qquad \qquad t=0.5\qquad\qquad t=1.0\qquad\qquad t=2.0\qquad\qquad t=3.0\qquad\qquad t=4.0\qquad \qquad \quad .\\
Upper: DTB solution; middle: reference solution; lower: solution error
\end{tabular}
\caption{2‑D Allen–Cahn equation solved with DTB Algorithm~\ref{alg: DTB for 2d ac}. The reference solution is obtained via a Fourier spectral method as in \cite{chen2024teng}. The DTB solution cleanly captures the dynamics, with uniformly small errors throughout the simulation.}
\label{fig: AC sol 2d}
\end{figure}

    \subsubsection{5-D case} We set $\nu=0.01$ and the initial condition to be the function defined in \eqref{eq: AC u0}. Notice that $w(x)\in [-1, 1] $ for any $x\in [-1, 1]^5$.
    
    Directly updating the network parameters $\theta$ at every time step as in~\eqref{eq: forward theta update} may become impractical in high dimensions or for strongly nonlinear right-hand sides. In these regimes, the DTB coefficients $\alpha^k$ obtained from the linear system \eqref{eq: G alpha sol} or \eqref{eq: J alpha solution} often exhibit large $\ell^2$ or $\ell^\infty$ norms, which in turn makes the parameter dynamics stiff. To preserve stability, we need to reduce the time step $h$ to an intractable small value, severely slowing the simulation. This phenomenon is demonstrated in the table presented in the numerical example in Section \ref{sec: num verify taylor error}.

    To circumvent the above stiffness we apply the DTB method Algorithm~\ref{alg: forward euler DTB} with the update rule given in Proposition \ref{prop: theta update opt} 
    We take $L=20$, that is we reinitialize the parameters $\theta$ every $20$ time step.   
    The DTB architecture is the same as that given in Section \ref{sec: num func approx}. We evaluate the solution in the five two-dimensional hyperplanes defined in Section~\ref{sec: num func approx} and display the snapshots in Fig.~\ref{fig: AC sol 5d}. 

    The experiments demonstrate that the DTB scheme yields stable and accurate parameter dynamics in regimes prone to stiffness, enabling practical time step size without excessive refinement. Across the five two‑dimensional evaluation planes, the snapshots in Fig. \ref{fig: AC sol 5d} show consistent solution quality and smooth evolution, indicating that the architecture from Section \ref{sec: num func approx} generalizes well to the higher‑dimensional setting. The observed behavior aligns with the anticipated impact of large DTB coefficient norms, but the DTB solver maintained robustness and accuracy throughout the computation. In general, these results support the DTB as a scalable approach for high‑dimensional nonlinear problems.

\begin{figure}[htbp]
\centering
\begin{tabular}{c}
\vspace{-18pt}  
\includegraphics[width=.95\linewidth]{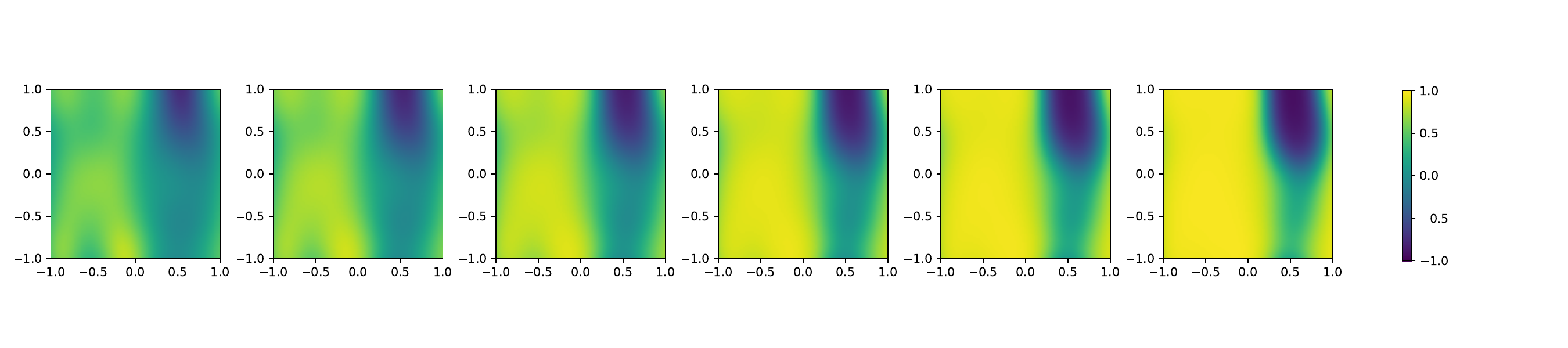}\\
\vspace{-18pt}  
\includegraphics[width=.95\linewidth]{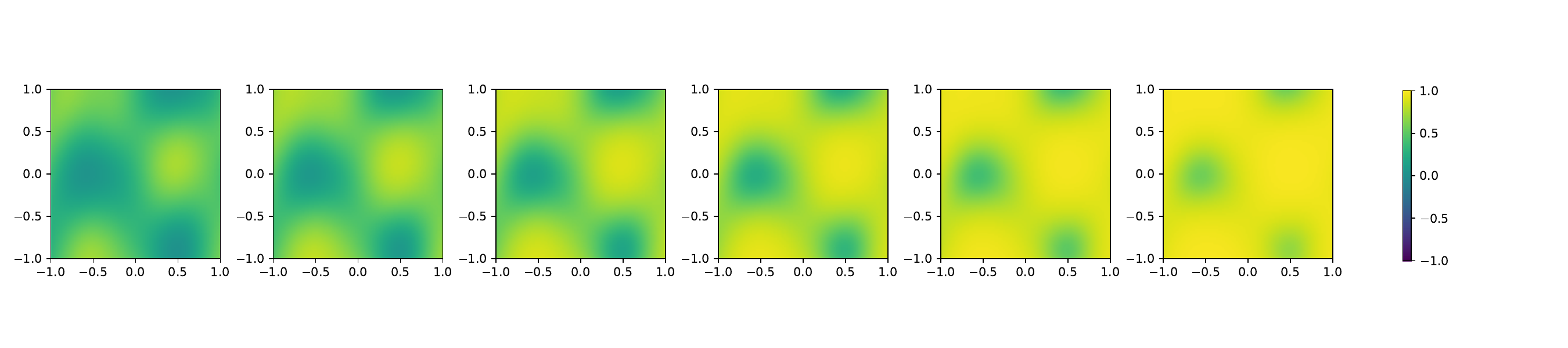} \\
\vspace{-18pt}  
\includegraphics[width=.95\linewidth]{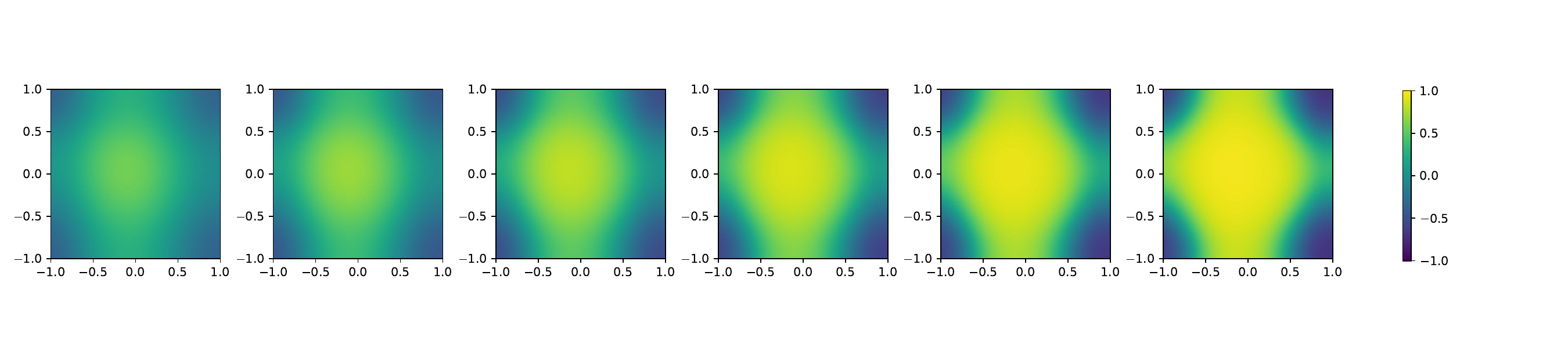}\\
\vspace{-18pt}  
\includegraphics[width=.95\linewidth]{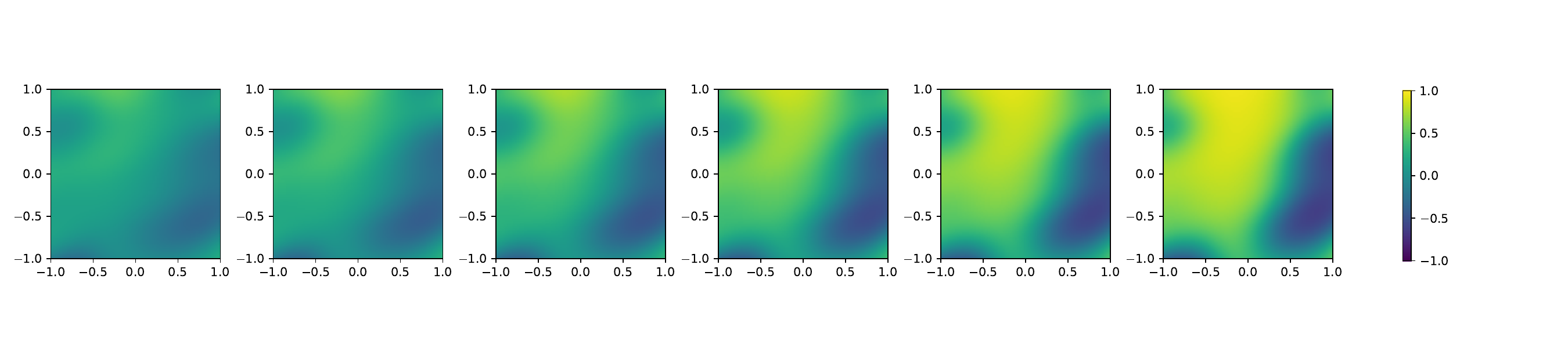}\\
\vspace{-18pt}  
\includegraphics[width=.95\linewidth]{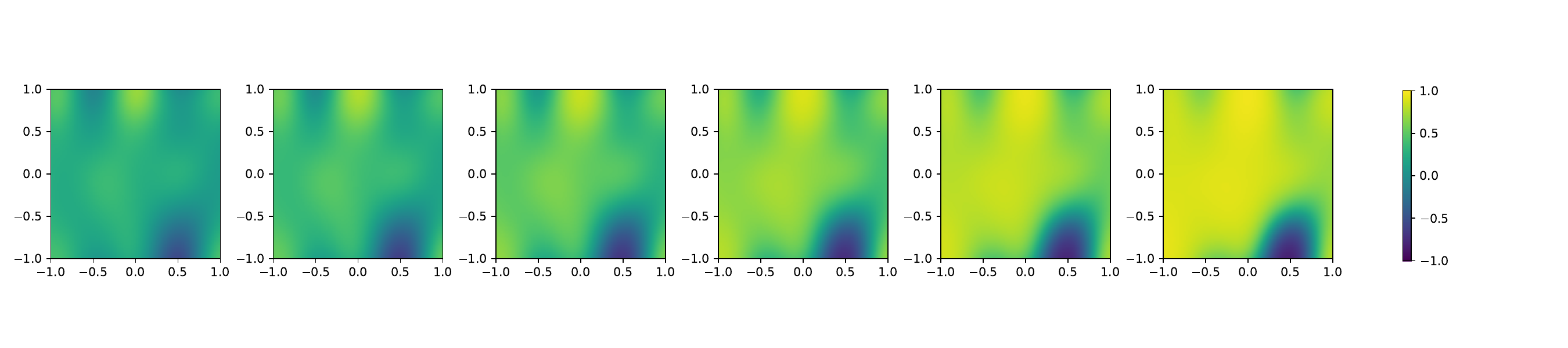}\\
\vspace{20pt}
t=0.\qquad \qquad t=0.4\qquad\qquad t=0.8\qquad\qquad t=1.2\qquad\qquad t=1.6\qquad\qquad t=2.0\qquad \qquad \quad .\\
The rows are solution projections onto subplane (a)-(e) 
\end{tabular}
\caption{5‑D Allen–Cahn equation solved with DTB Algorithm~\ref{alg: forward euler DTB}, where the parameters is updated as in Proposition~\ref{prop: theta update opt}. Evaluation is performed on the five two-dimensional hyperplanes (a)-(e). The solution shows the expected phase separation and concentration behavior of Allen–Cahn dynamics in high dimension.}
\label{fig: AC sol 5d}
\end{figure}

\subsection{Wasserstein flows}\label{sec: wf numerical} Here, we present the numerical results of the WGF and WHF as discussed in section \ref{sec: WF intro}.

\subsubsection{WGF with interaction potential.}\label{sec: num wgf}
We first consider solving the WGF with interaction potential \eqref{eq: interact-F} where:
\begin{align}
    \label{eq: wgf agg kernel}
    J(z)=\frac{|z|^4}{4}-\frac{|z|^2}{2},
\end{align}
and the initial density is set to be Gaussian distribution with mean $\mu_0=(1.25, 1.25)^{\top}$ and variance $\gamma I$ where $\gamma=0.6$:
\begin{equation}\label{eq: wgf agg initial density}
    \rho_0(z)=\frac{1}{\sqrt{2\pi}\gamma }e^{-\frac{|z-\mu_0|^2}{2\gamma^2}}.
\end{equation}
Given the initial condition, the solution \(\rho_t\) converges to the steady-state solution \(\rho_*\), which is a Dirac distribution uniformly concentrated on a ring with radius \(0.5\) centered at \(\mu_0\). We apply the DTB method to approximate the evolution equation \eqref{def: operator dyn wgf} of the push-forward map $T$ with DNN $T_{\theta}$. The same neural network architecture (residual neural networks with 2 hidden layers and 50 neurons in each hidden layer) as in \cite{JIN2025113660} is used for both the DTB and PWGF methods. Numerical results are shown in Figure~\ref{aggregation sampleplot}. The first row corresponds to the PWGF method, while the second row corresponds to the DTB method.

When approaching the steady-state map as \(t \to \infty\), the true push-forward map should transform the initial Gaussian distribution (supported on a single connected domain, i.e., \(\mathbb{R}^2\)) into a Dirac distribution concentrated on the ring (a domain with a hole). However, such a transformation is not possible with a continuous function due to the topological structure.

As \(t\) increases, the PWGF struggles to learn this push-forward map, and ultimately the solution converges to a polygon-shaped region instead of the desired ring. In contrast, the DTB method learns a much smoother and more accurate circular structure, closely resembling the true steady-state solution.

\begin{figure*}[t!]
    \begin{subfigure}{0.16\textwidth}
        \centering
        \includegraphics[width=0.99\linewidth]{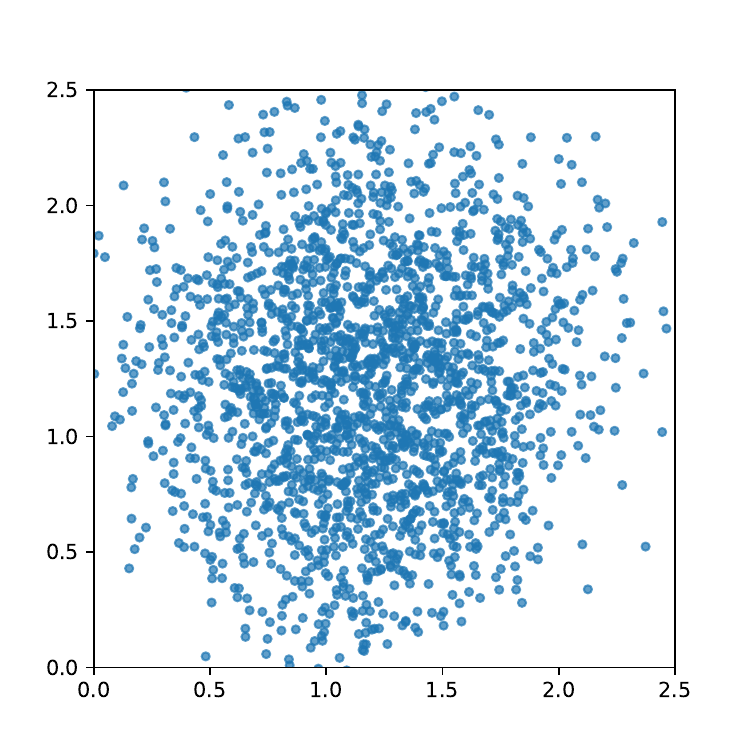}
        \caption{$t=0$}
    \end{subfigure}%
    \begin{subfigure}{0.16\textwidth}
        \centering
        \includegraphics[width=0.99\linewidth]{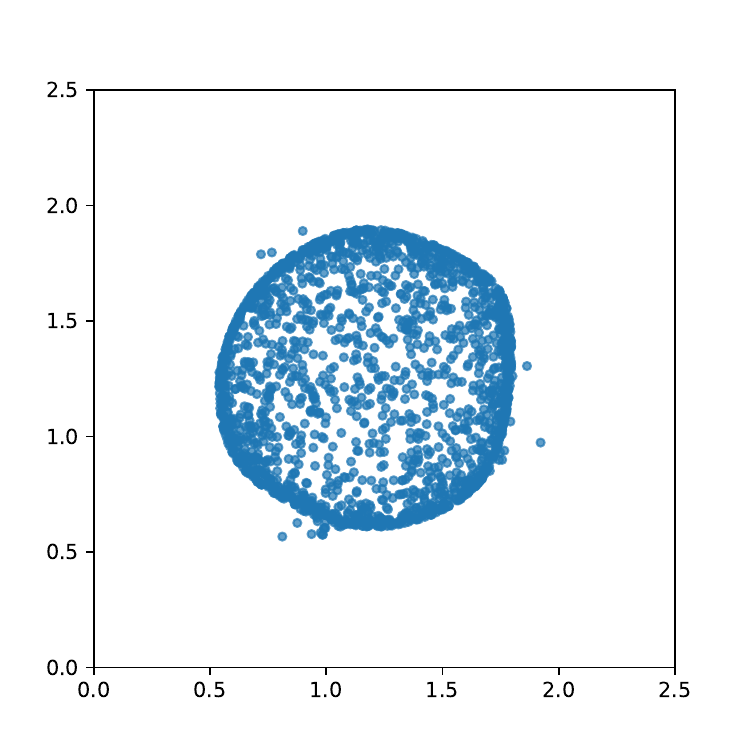}
        \caption{$t=3$}
    \end{subfigure}
    \begin{subfigure}{0.16\textwidth}
        \centering
        \includegraphics[width=0.99\linewidth]{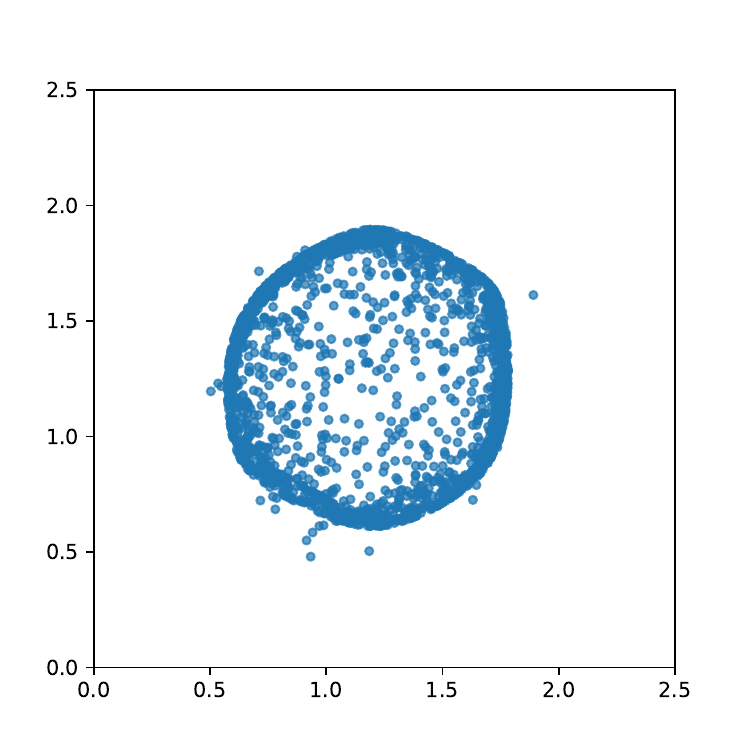}
        \caption{$t=6$}
    \end{subfigure}
    \begin{subfigure}{0.16\textwidth}
        \centering
        \includegraphics[width=0.99\linewidth]{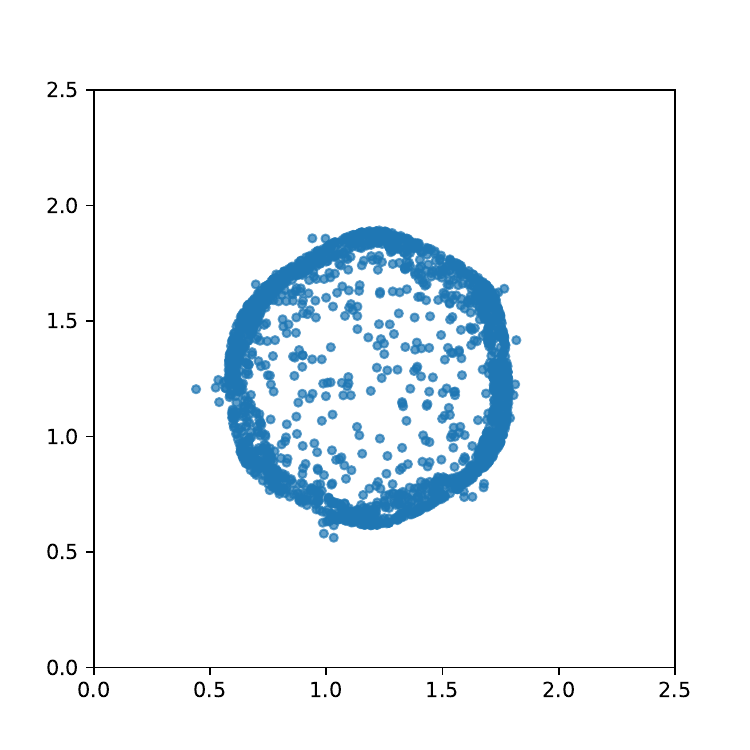}
        \caption{$t=9$}
    \end{subfigure}%
    \begin{subfigure}{0.16\textwidth}
        \centering
        \includegraphics[width=0.99\linewidth]{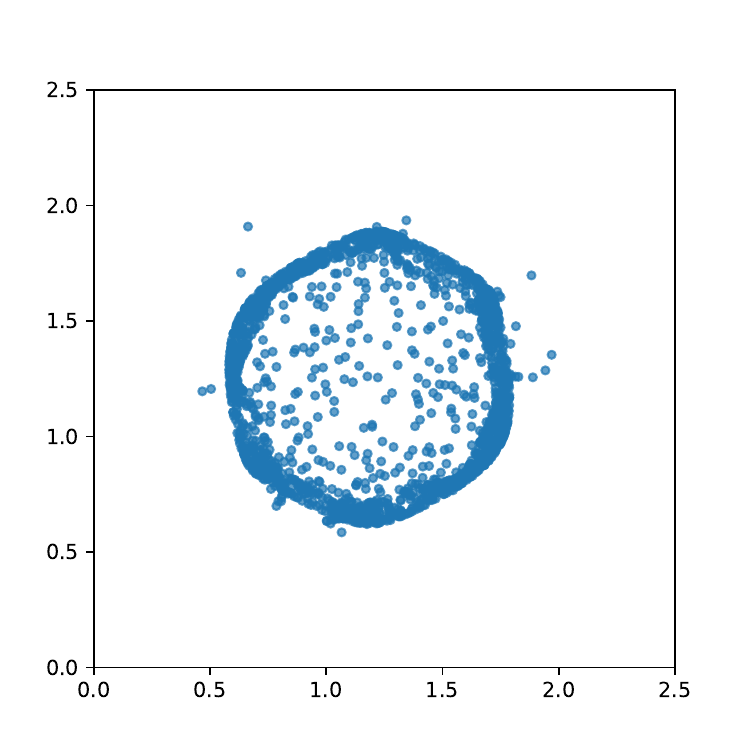}
        \caption{$t=12$}
    \end{subfigure}
    \begin{subfigure}{0.16\textwidth}
        \centering
        \includegraphics[width=0.99\linewidth]{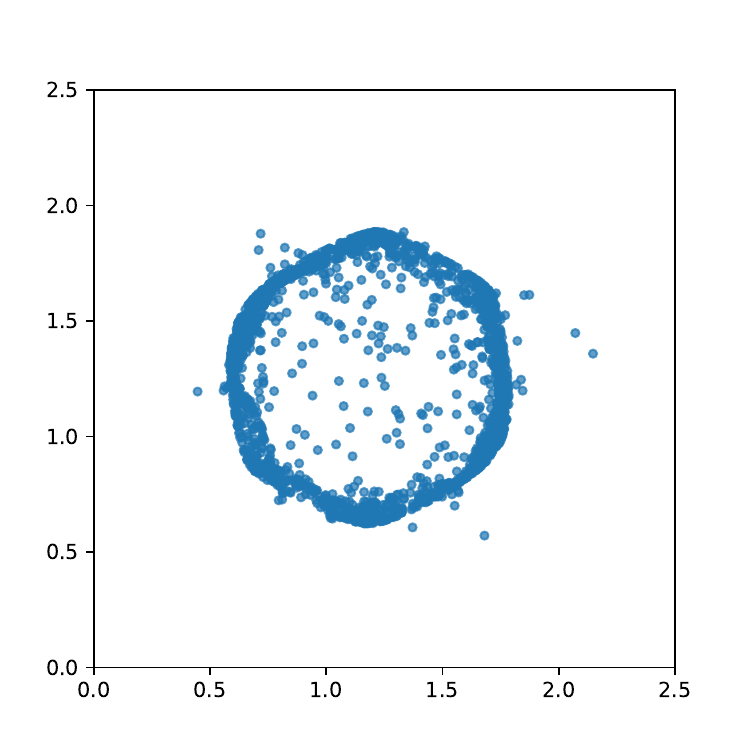}
        \caption{$t=15$}
    \end{subfigure}
        \begin{subfigure}{0.16\textwidth}
        \centering
        \includegraphics[width=0.99\linewidth]{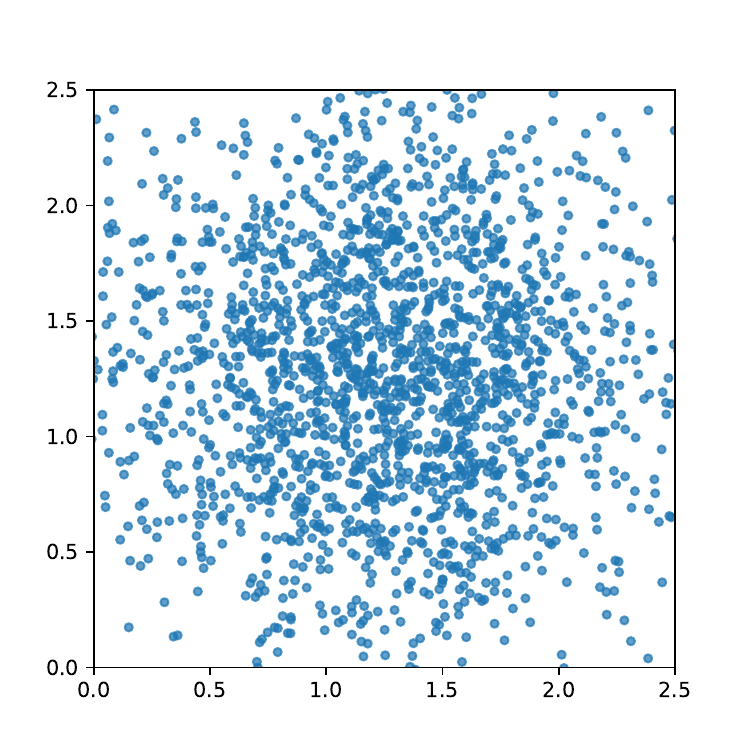}
        \caption{$t=0$}
    \end{subfigure}%
    \begin{subfigure}{0.16\textwidth}
        \centering
        \includegraphics[width=0.99\linewidth]{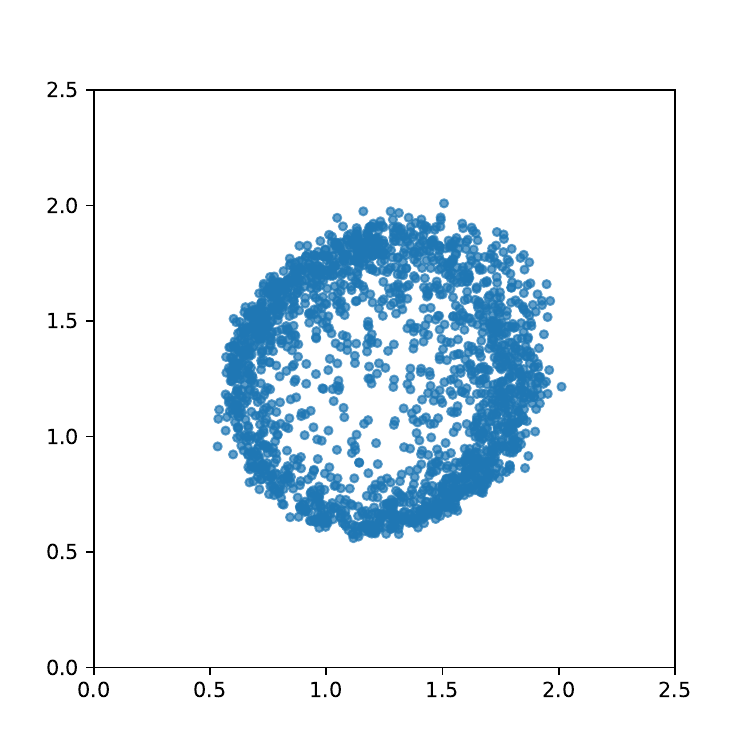}
        \caption{$t=3$}
    \end{subfigure}
    \begin{subfigure}{0.16\textwidth}
        \centering
        \includegraphics[width=0.99\linewidth]{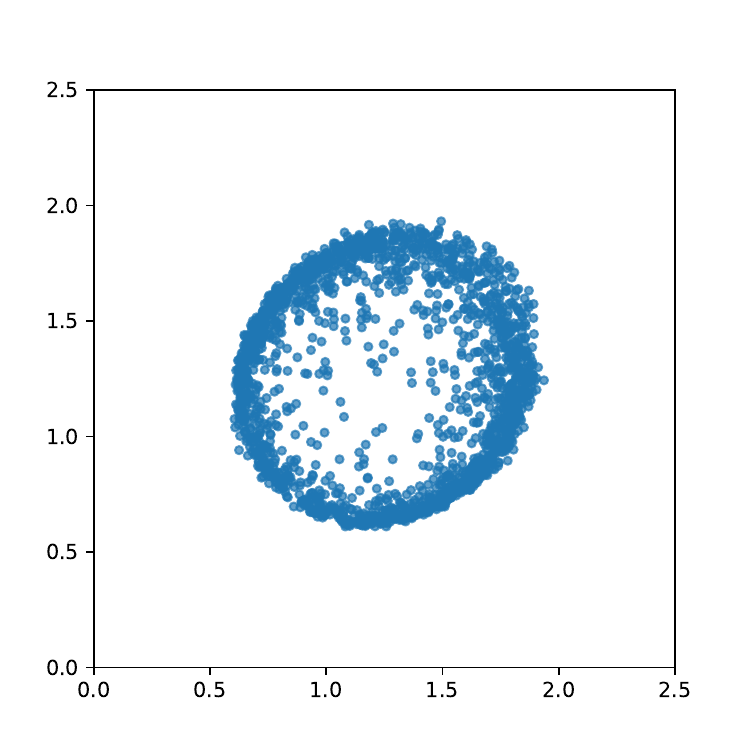}
        \caption{$t=6$}
    \end{subfigure}
    \begin{subfigure}{0.16\textwidth}
        \centering
        \includegraphics[width=0.99\linewidth]{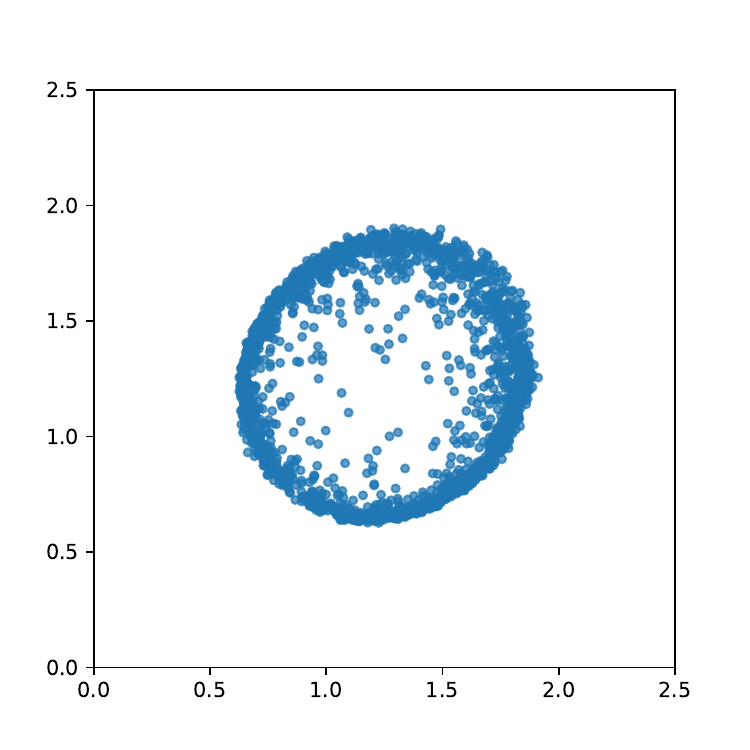}
        \caption{$t=9$}
    \end{subfigure}%
    \begin{subfigure}{0.16\textwidth}
        \centering
        \includegraphics[width=0.99\linewidth]{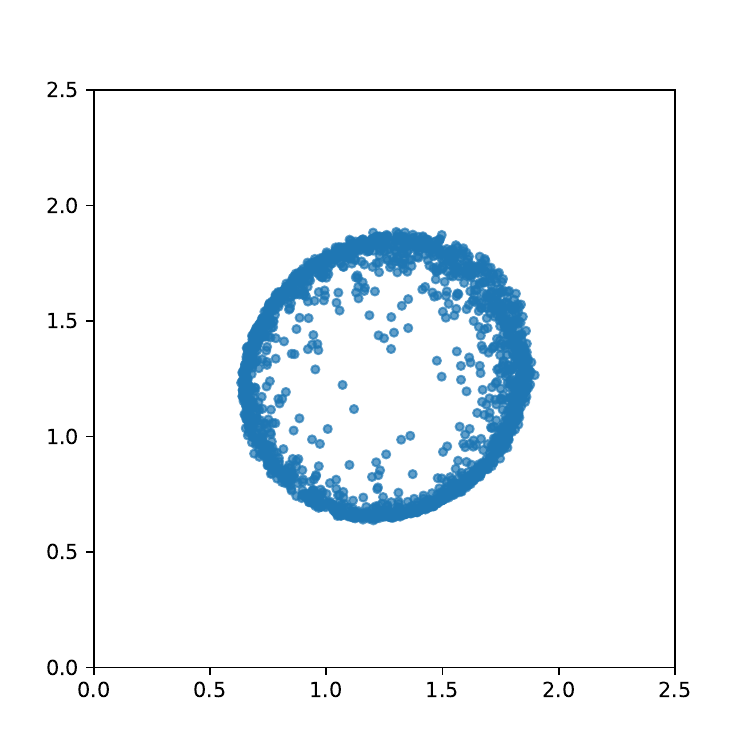}
        \caption{$t=12$}
    \end{subfigure}
    \begin{subfigure}{0.16\textwidth}
        \centering
        \includegraphics[width=0.99\linewidth]{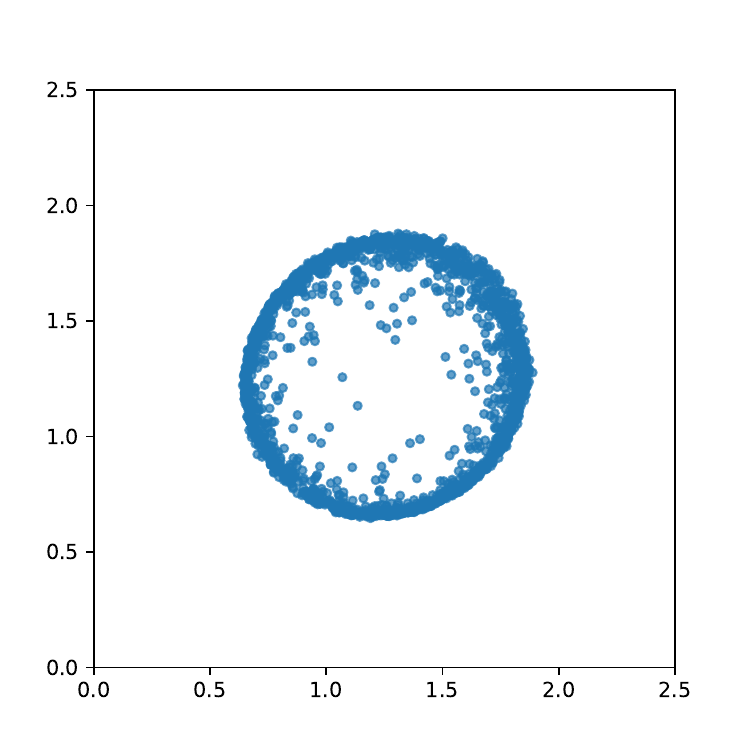}
        \caption{$t=15$}
    \end{subfigure}
    \caption{Sample plots of computed $\rho_{\theta}$ at different time $t$ for WGF with the interaction potential. Top row: PWGF solution; Bottom row: DTB solution}
    \label{aggregation sampleplot}
\end{figure*}

\subsubsection{Harmonic Oscillator as WHF}\label{sec: num whf ho}
Another example is the $10D$ WHF with quadratic potential \eqref{eq: whf linear op} where 
\begin{align}
    &V(z) = {\frac{3}{8}z_1^2+\frac{1}{2}\sum_{i=2}^{10}z_i^2 },\\
    &\rho(0, z) = \frac{1}{2}\sum_{i=2}^{10}z_i^2.
\end{align}
In this case, the trajectory of any randomly chosen initial point can be explicitly solved \cite{doi:10.1137/23M159281X}, and the solution corresponds to a harmonic oscillator.

We apply both the PWHF and DTB methods using the same neural network structure (residual neural networks with 2 hidden layers and 80 neurons in each hidden layer) and the same time discretization step size. For the DTB method, we use Algorithm \ref{alg: G-DTB} to approximate the right-hand side of \eqref{eq: whf linear op} through a DNN. The second-order dynamics are rewritten as a system of first-order equations, which are then integrated using the forward Euler scheme.
Figure \ref{fig: L2 HO} shows the relative $L^2$ error for the calculated solutions, and figure \ref{fig: HO trajectory sampleplot} shows the trajectories of random picked initial points under the push-forward map evolution. The DTB method provides better approximations to the trajectories.
\begin{figure*}
    \begin{subfigure}{0.3\textwidth}
        \centering
        \includegraphics[width=0.99\linewidth]{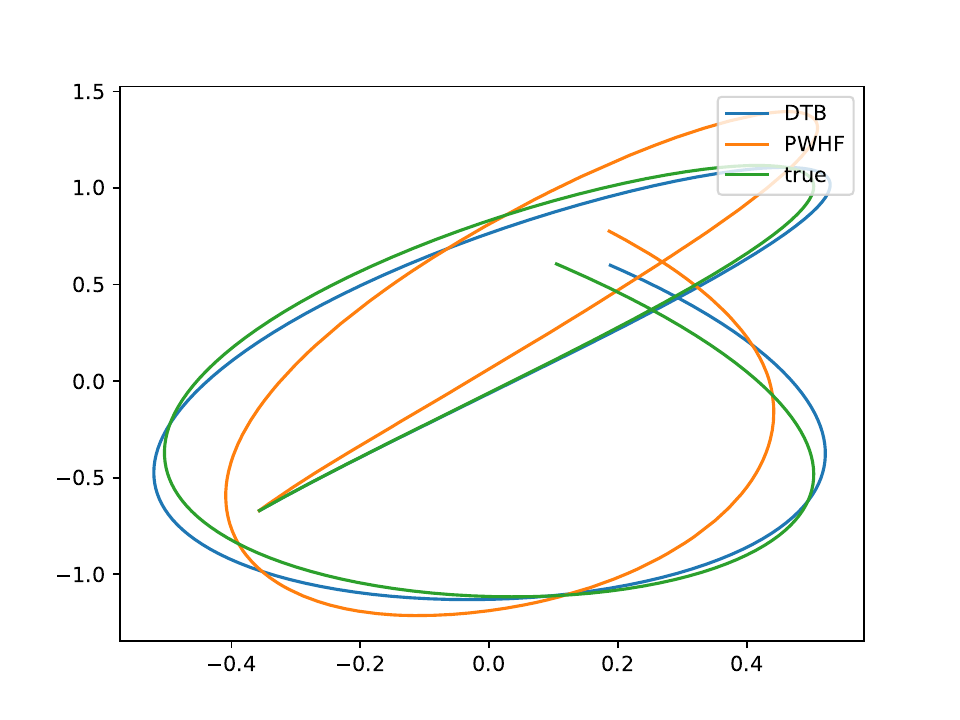}
        \caption{sample $1$}
    \end{subfigure}%
    \begin{subfigure}{0.3\textwidth}
        \centering
        \includegraphics[width=0.99\linewidth]{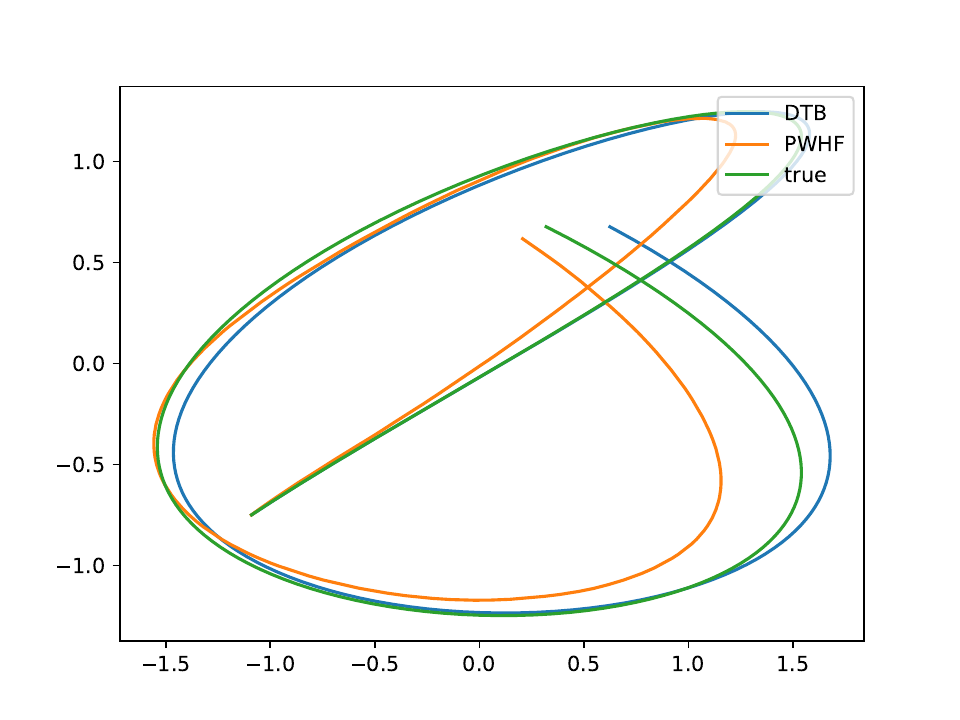}
        \caption{sample $2$}
    \end{subfigure}
    \begin{subfigure}{0.3\textwidth}
        \centering
        \includegraphics[width=0.99\linewidth]{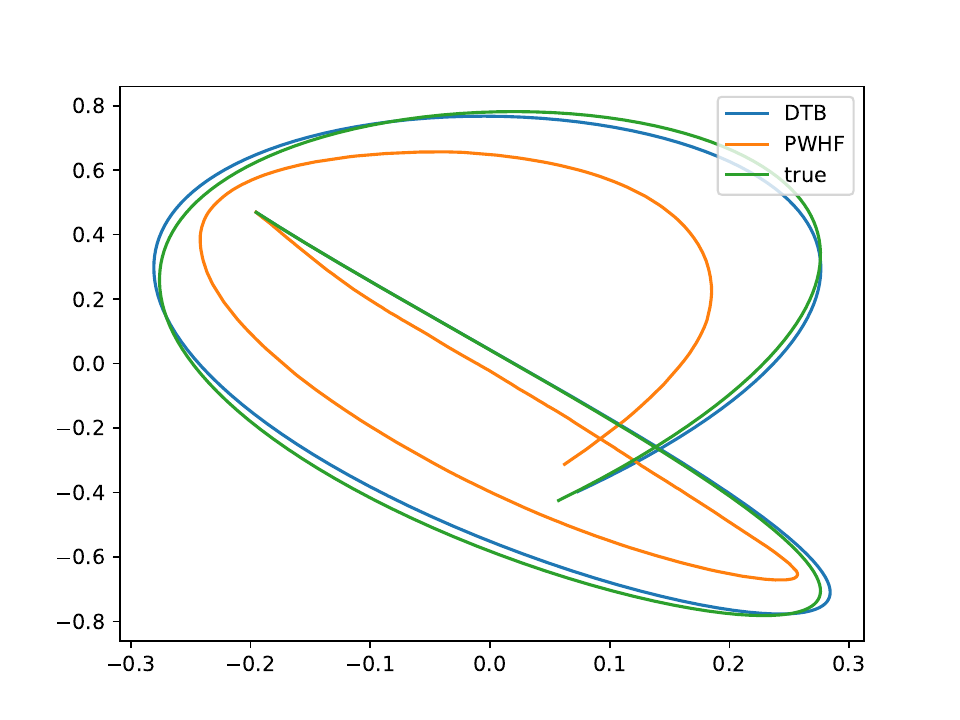}
        \caption{sample $3$}
    \end{subfigure}
    \\
    \begin{subfigure}{0.3\textwidth}
        \centering
        \includegraphics[width=0.99\linewidth]{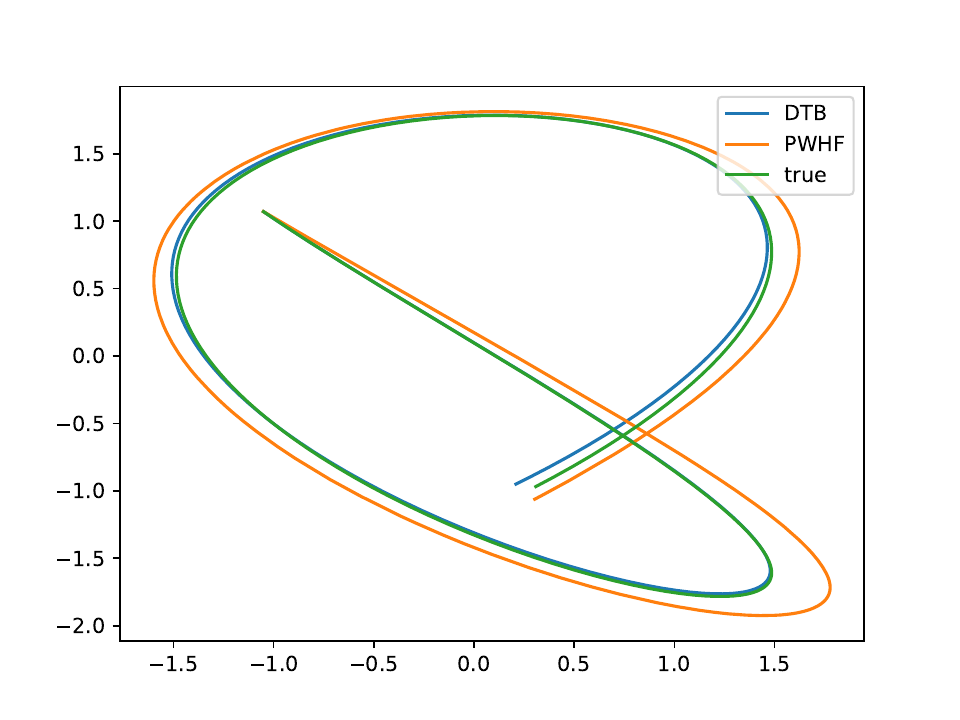}
        \caption{sample $4$}
    \end{subfigure}%
    \begin{subfigure}{0.3\textwidth}
        \centering
        \includegraphics[width=0.99\linewidth]{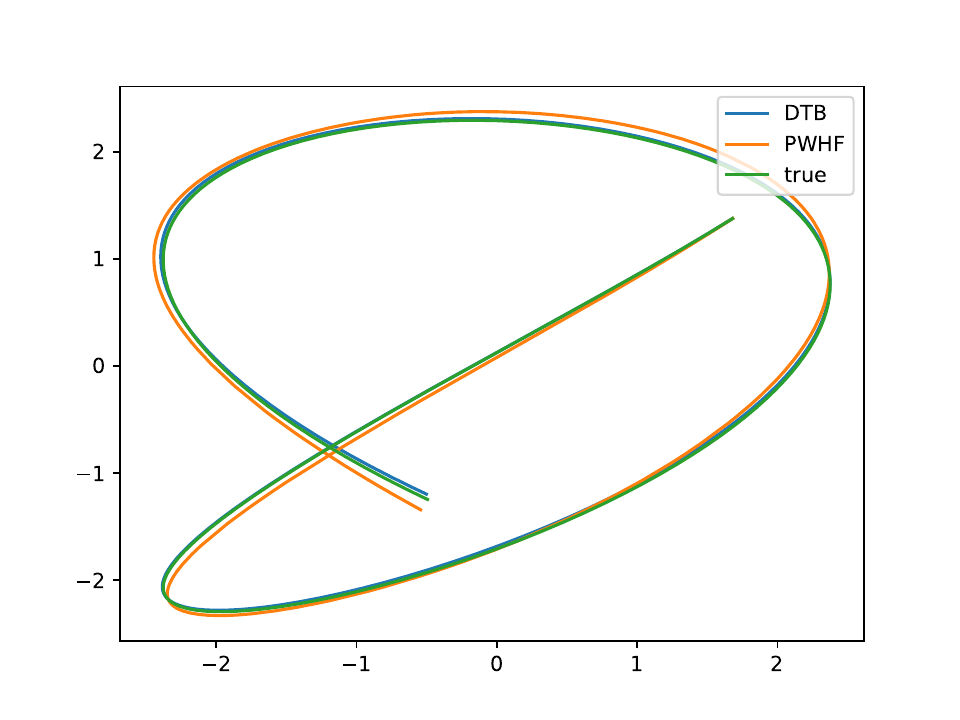}
        \caption{sample $5$}
    \end{subfigure}
    \begin{subfigure}{0.3\textwidth}
        \centering
        \includegraphics[width=0.99\linewidth]{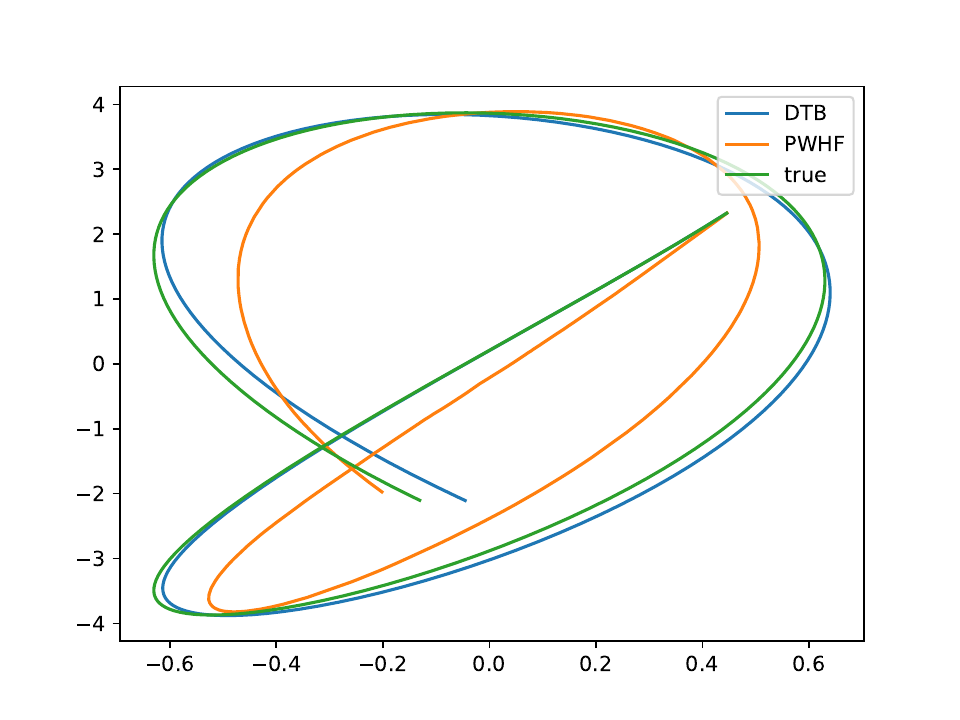}
        \caption{sample $6$}
    \end{subfigure}
    
    \caption{Trajectory plots of random picked initial points for the harmonic oscillator example.}
    \label{fig: HO trajectory sampleplot}
\end{figure*}

\begin{figure}[H]
    \centering
    \includegraphics[width=0.5\linewidth]{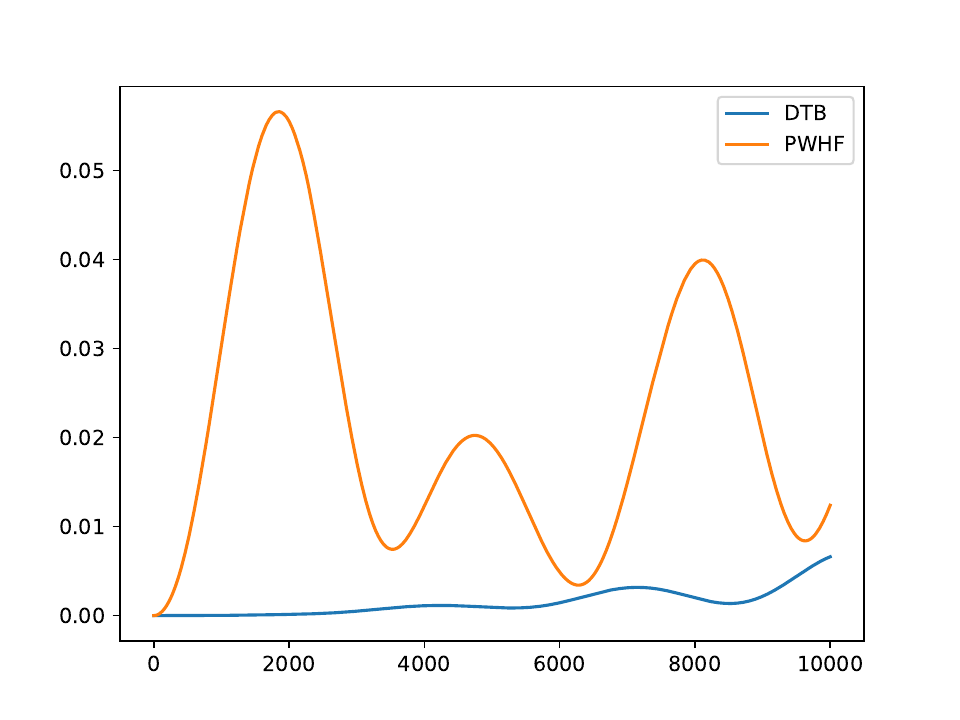}
    \caption{relative $L^2$ error versus time for the harmonic oscillator example. Under the same settings, the DTB attains consistently lower relative errors than the PWHF.}
    \label{fig: L2 HO}
\end{figure}

\section{Conclusion}

We introduce a meshless and training-free method that leverages the traditional temporal schemes and the tangent bundle of DNN to compute the solution of evolutionary PDEs in high dimension. Compared to some of the existing LIT methods, the DTB method exhibits noticeable improvements in accuracy with enhanced stability and robustness. Furthermore, it offers convenient adaptivity to combine multiple DNNs and theoretically provides error control. Meanwhile, several problems remain unanswered. For example, are there more systematic strategies or guidelines in updating neural network parameters? Our current recommendations, as discussed in subsection on adapting DNN parameters, are suggested heuristically. A principled method is more desirable in practice. How does the solution error quantitatively depend on the structure of DNN, the sample size,  and the subspace selection in the DTB bases? Can we extend the DTB method to solve elliptic PDEs? Those are questions worthy of further investigation.

\section{Acknowledgments} This research is partially supported by NSF grants DMS-2307465 and DMS-2510829.



\bibliographystyle{siamplain}
\bibliography{references}

\appendix

\section{Algorithm design for 2-D Allen-Cahn equation}\label{sec: ac 2d alg}
To enhance the accuracy when solving the 2-D Allen-Cahn equation, we modify Algorithm \ref{alg: forward euler DTB} by incorporating an additional correction step. The forward update rule for DTB base selection is also employed, leading to the following algorithm.

\begin{algorithm}[H]
\caption{DTB method with correction for solving $2$-D Allen-Cahn equation}
\label{alg: DTB for 2d ac}
\begin{algorithmic}[1]
\STATE{Set terminal time $t_1$ and number of steps $K$. Set step size $h=t_1/K$.}
\STATE{Initialize $DTBset=\emptyset$ for solved DTB approximation. Initialize solution $u^0=u(0, x)$}
\FOR{$k=0, \dots, K-1$}
\STATE{Generate samples $z^j$ from some reference distribution $\lambda$ for $j=1,\dots,N_{\theta}$}.
\STATE{Choose random subset $\mathcal{B}_i\subset \mathcal{B}(f, \theta^k), \ i=1, 2$, and construct the metric tensor  $G_i(\theta^k)$ of $\mathcal{B}_i$}
\STATE{Apply Algorithm \ref{alg: J-DTB} to solve $\alpha^1$ from $G_1(\theta^k)\alpha^1 =p^\textrm{temp}$, where $p^\textrm{temp}=P(\theta^k; F[\widehat{u}^k])$}
\STATE{Solve $\alpha^2$ from $G_2(\theta^k)\alpha^2 =P(\theta^k; R^\textrm{temp})$ where the residual $R^\textrm{temp}(z)=F[u^k(z)]-\partial_{\theta}f_{\theta^k}(z)\alpha^1$}
\STATE{Set temporary solution $u^{\textrm{temp}}(z)=u^{k}(z)+h \cdot \frac{\partial}
{\partial\theta}f_{\theta^k}(z) \cdot( \alpha^1+\alpha^2)$}
\STATE{Solve $\alpha^3$ from $G_1(\theta^k)\alpha^3 =p^\textrm{new}$ where the new estimate $p^\textrm{new}=P(\theta^k; \frac{1}{2}(F[u^k]+F[u^{\textrm{temp}}]))$}
\STATE{Solve $\alpha^4$ from $G_2(\theta^k)\alpha^4 =P(\theta^k; R^\textrm{new})$ where $R^\textrm{new} =  \frac{1}{2}(F[u^k]+F[u^{\textrm{temp}}]) - \partial_{\theta}f_{\theta^k}(z)\alpha^3$}
\STATE{Update the parameter $\theta^{k+1}=\theta^k + h\alpha^3+h\alpha^4$; Update $u^{k+1}(z)=u^{k}(z)+h \cdot \frac{\partial}{\partial\theta}f_{\theta^k}(z) \cdot( \alpha^3+\alpha^4)$}
\STATE{Update $\textrm{DTBset}+=\{(\theta^{k}, \alpha^3+\alpha^4)\}$}
\ENDFOR
\STATE{\textbf{Output}: Solution $u^K$ and $\textrm{DTBset}$. }
\end{algorithmic}
\end{algorithm}

\section{Details of the implementation for WGF/WHF}\label{appendix: wgf_details}
For problems such as WGF and WHF, the underlying geometry of the probability manifold evolves with the density $\rho(t, z)$ itself. Therefore, instead of the standard $L^2$ metric, we adopt a time-dependent $\rho$-weighted $L^2$ metric. This helps the proposed geometric framework correctly capture the evolving structure of the problems.

This choice naturally modifies the definitions of the metric tensor $G$ and the projection operator $P$ to incorporate the density $\rho$ as a weight. At each time step, these are defined as:
\begin{align}
    G(\theta, \rho) &= \mathbb{E}_{\mathbf{X}\sim \rho}\left[\frac{\partial}{\partial\theta}f_{\theta}(\mathbf{X})^\top \frac{\partial}{\partial\theta}f_{\theta}(\mathbf{X})\right], \label{eq:metric_expectation} \\
    P(\theta, \rho; g) &= \mathbb{E}_{\mathbf{X}\sim \rho}\left[\frac{\partial}{\partial\theta}f_{\theta}(\mathbf{X})^\top g(\mathbf{X})\right]. \label{eq:proj_expectation}
\end{align}

The function level projection operator can be modified as well:
\begin{align}
    \label{eq: func opt approx rho weighted}
    \mathcal{K}_{\theta,\rho}[g](x) = \frac{\partial}{\partial\theta}f_{\theta}(x)G(\theta, \rho)^{\dagger} P(\theta, \rho; g).
\end{align}

In the Wasserstein gradient flow induced by an energy $\mathcal{F}(\rho)$, the projection of $g(x)=\nabla_X\frac{\delta}{\delta\rho}\mathcal{F}(\rho, x)$ through the operator $P$ can be computed with a single backward pass using automatic differentiation.
\begin{align}
    P(\theta, \rho; \nabla_X\frac{\delta}{\delta\rho}\mathcal{F}(\rho, \cdot))&=\mathbb{E}_{\mathbf{X}\sim \rho}\left[\frac{\partial}{\partial\theta}f_{\theta}(\mathbf{X})^\top \nabla_X\frac{\delta}{\delta\rho}\mathcal{F}(\rho, \mathbf{X})\right]\\
    &=\nabla_{\theta} \mathbb{E}_{\mathbf{X}\sim\rho}\left[f_{\theta}(\mathbf{X})\cdot \nabla_X\frac{\delta}{\delta\rho}\mathcal{F}(\rho, \mathbf{X})\right].
\end{align}
For the energy functional discussed in this paper, the projections are:
\begin{itemize}
    \item For the interaction potential \eqref{eq: interact-F}, we have
    \begin{align}
        P(\theta, \rho; \nabla_X\frac{\delta}{\delta\rho}\mathcal{F}(\rho, \cdot))&=\nabla_{\theta} \mathbb{E}_{\mathbf{X}\sim\rho}\left[f_{\theta}(\mathbf{X})\cdot \int J(|x-y|)\rho(y)dy\right]\nonumber\\
        &=\nabla_{\theta} \mathbb{E}_{\mathbf{X}\sim\rho}\left[f_{\theta}(\mathbf{X})\cdot \mathbb{E}_{\mathbf{Y}\sim\rho} J(|\mathbf{X}-\mathbf{Y}|)\right]
    \end{align}
    \item For the linear potential $ \mathcal{F}(\rho) = \int V(z)\rho(z)dx$, we have
    \begin{align}
        P(\theta, \rho; \nabla_X\frac{\delta}{\delta\rho}\mathcal{F}(\rho, \cdot))&=\nabla_{\theta} \mathbb{E}_{\mathbf{X}\sim\rho}\left[f_{\theta}(\mathbf{X})\cdot \nabla V(\mathbf{X})\right]
    \end{align}
\end{itemize}

In practice, expectations are estimated via Monte Carlo. At each step $k$, samples from the target density $\rho^k$ are generated by applying the current numerical map $T^k$, which is defined by the iteration in \eqref{eq: wgf particle iteration}, to a set of base samples $\{\mathbf{Z}_i\}$ drawn from the reference distribution $\lambda$.

\end{document}